\documentclass[12pt, reqno]{amsart}
\usepackage{tikz}
\usepackage{amssymb, mathrsfs}
\usepackage{latexsym}
\usepackage{amsmath}
\usepackage{amsthm}
\usepackage{amsfonts}
\usepackage{dsfont}
\usepackage{mathtools}
\usepackage{epsfig}
\usepackage{verbatim}
\usepackage[all,cmtip]{xy}
\usepackage[left=1.2in,right=1.2in,top=1in,bottom=1in]{geometry}
\usepackage{enumerate}
\usepackage[colorlinks=true,linkcolor=blue,citecolor=blue]{hyperref}

\usepackage{bm}

\usepackage{tikz}

\newtheorem*{lemma*}{Lemma}

\newtheorem{theorem}{Theorem}[section]
\newtheorem{proposition}[theorem]{Proposition}
\newtheorem{lemma}[theorem]{Lemma}

\newtheorem{corollary}[theorem]{Corollary}

\newtheorem{question}{Question}
\newtheorem*{question*}{Question}

\theoremstyle{definition}

\newtheorem*{claim*}{Claim}
\newtheorem{definition}[theorem]{Definition}

\theoremstyle{remark}
\newtheorem{remark}[theorem]{Remark}

\def\Xint#1{\mathchoice
	{\XXint\displaystyle\textstyle{#1}}%
	{\XXint\textstyle\scriptstyle{#1}}%
	{\XXint\scriptstyle\scriptscriptstyle{#1}}%
	{\XXint\scriptscriptstyle\scriptscriptstyle{#1}}%
	\!\int}
\def\XXint#1#2#3{{\setbox0=\hbox{$#1{#2#3}{\int}$ }
		\vcenter{\hbox{$#2#3$ }}\kern-.6\wd0}}

\def\dashint{\Xint-}

\newcommand{\ind}{\textup{Ind}}
\newcommand{\supp}{\textup{supp}}
\newcommand{\ev}{\textup{ev}}

\newcommand{\smooth}{\mathscr R}
\newcommand{\tr}{\mathrm{tr}}
\newcommand{\nDelta}{\varDelta}

\begin{document}
\title[Delocalized eta, algebraicity and $K$-theory]{Delocalized eta invariants, algebraicity, and $K$-theory of group $C^*$-algebras}

\author{Zhizhang Xie}
\address[Zhizhang Xie]{Department of Mathematics, Texas A\&M University}
\email{xie@math.tamu.edu}
\thanks{The first author is partially supported by NSF 1500823, NSF 1800737.}

\author{Guoliang Yu}
\address[Guoliang	 Yu]{Department of Mathematics, Texas A\&M University}
\email{guoliangyu@math.tamu.edu}
\thanks{The second author is partially supported by NSF 1700021.}

\maketitle

\begin{abstract}
In this paper, we establish a precise connection between higher rho invariants and delocalized eta invariants. Given an element in a discrete group, if its conjugacy class  has polynomial growth, then there is a natural trace map on the $K_0$-group of its group $C^\ast$-algebra. For each such trace map, we construct a determinant map on secondary higher invariants.  We show that, under the evaluation of this determinant map,  the image of a higher rho invariant is precisely the corresponding delocalized eta invariant of Lott.  As a consequence,  we show that if the Baum-Connes conjecture holds for a group, then Lott's delocalized eta invariants take values in algebraic numbers. We also generalize Lott's delocalized eta invariant to the case where the corresponding conjugacy class does \emph{not} have polynomial growth, provided that the strong Novikov conjecture holds for the group. 
\end{abstract}

\section{Introduction}

Let $X$ be a complete manifold  of dimension $n$ with a discrete group $\Gamma$ acting on it properly and cocompactly through isometries.  Each $\Gamma$-equivariant elliptic differential operator $D$ on $X$ gives rise to a higher index class  $\ind_\Gamma(D) \in K_n(C_r^\ast(\Gamma))$. This higher index is an obstruction to the invertibility  of $D$. It is a far-reaching generalization of the classical Fredholm index and plays a fundamental role in the studies of many problems in geometry and topology such as the Novikov conjecture, the Baum-Connes conjecture and the Gromov-Lawson-Rosenberg conjecture.  Higher index classes are often referred to as primary invariants. When the higher index class of an operator is trivial and given a specific trivialization,  a secondary index theoretic invariant naturally arises. One such example is the associated  Dirac operator $\widetilde D$ on the universal covering $\widetilde M$ of a closed spin manifold $M$, which is equipped with a positive scalar curvature metric $g$. In this case, it follows from the Lichnerowicz formula that the higher index of the Dirac operator vanishes. And there is a natural secondary higher invariant of $\widetilde D$ -- introduced by Higson and Roe  \cite{MR2220522,MR2220523,MR2220524, MR1399087} -- called the higher rho invariant of $\widetilde D$ (with respect to the metric $g$),  cf. Section $\ref{sec:roelocal}$ and  Section $\ref{sec:deloc}$ below for details.  This higher rho invariant is an obstruction to the inverse of the Dirac operator being local, and has important applications to geometry and topology.

On the other hand, for the same Dirac operator $\widetilde D$ above,  Lott introduced the following delocalized eta invariant $\eta_{\langle h\rangle}(\widetilde D)$ \cite{MR1726745}: 
\begin{equation}
\eta_{\langle h\rangle}(\widetilde D) \coloneqq  \frac{2}{\sqrt \pi} \int_{0}^{\infty}  \tr_h(\widetilde D e^{-t^2 \widetilde D^2})dt,  
\end{equation} 
 under the condition that the conjugacy class $\langle h\rangle$ of $h\in \pi_1M$ has polynomial growth. Here  $\pi_1 M$ is the fundamental group of $M$, and   $\tr_h$ is the following trace map (see Section $\ref{sec:deloc}$ for more details):   \[ \tr_h(A) = \sum_{g\in \langle h \rangle} \int_{\mathcal F} A(x, gx) dx   \] 
on $\Gamma$-equivariant Schwartz kernels $A\in C^\infty(\widetilde M\times \widetilde M)$, where $\mathcal F$ is a fundamental domain of $\widetilde M$ under the action of $\Gamma$.

In this paper, we shall devise a conceptual $K$-theoretic approach to establish a precise connection between Higson-Roe's $K$-theoretic higher rho invariants  and Lott's delocalized eta invariants.  More precisely, we have the following theorem. 
\begin{theorem}\label{thm:intro}
	Let $M$ be a closed odd-dimensional spin manifold equipped with a positive scalar curvature metric $g$. Suppose $\widetilde M$ is the universal cover of $M$, $\tilde g$ is the Riemannian metric on $\widetilde M$  lifted from $g$, and $\widetilde D$ is the associated Dirac operator on $\widetilde M$. Suppose the conjugacy class $\langle h\rangle$ of a non-identity element $h\in \pi_1M$ has polynomial growth,   then we have
	\[ \tau_h (\rho(\widetilde D, \widetilde g) ) = -\frac{1}{2}\eta_{\langle h\rangle}(\widetilde D),   \]
	where $\rho(\widetilde D, \widetilde g)$ is the $K$-theoretic  higher rho invariant of $\widetilde D$ with respect to the metric $\tilde g$, and $\tau_h$ is a canonical determinant map associated to $\langle h \rangle$. 
\end{theorem}

While the definition of Lott's delocalized eta invariant requires certain growth conditions on $\pi_1 M$ (e.g. polynomial growth on a conjugacy class), the $K$-theoretic higher rho invariant can be defined in complete generality, without any growth conditions on $\pi_1 M$. We shall show how to generalize Lott's delocalized eta invariant without imposing any growth conditions on $\pi_1 M$, provided that the strong Novikov conjecture holds for $\pi_1 M$. This is achieved by using the Novikov rho invariant introduced in \cite[Section 7]{Weinberger:2016dq}.

As an application of Theorem $\ref{thm:intro}$ above, we have the following algebraicity result concerning the values of delocalized eta invariants.

\begin{theorem}
	With the same notation as above, if the rational Baum-Connes conjecture holds for $\Gamma$, and the conjugacy class $\langle h\rangle$ of a non-identity element $h\in \Gamma$ has polynomial growth,  then
	the delocalized eta invariant $\eta_{\langle h\rangle}(\widetilde D)$ is an algebraic number.  Moreover, if in addition $h$ has infinite order, then  $\eta_{\langle h\rangle}(\widetilde D)$ vanishes.   
\end{theorem}

This theorem follows from the construction of the determinant map $\tau_h$ and a $L^2$-Lefschetz fixed point theorem of B.-L. Wang and H. Wang  \cite[Theorem 5.10]{MR3454548}.  When $\Gamma$ is torsion-free and satisfies the Baum-Connes conjecture,  and the conjugacy class $\langle h\rangle$ of a non-identity element $h\in \Gamma$ has polynomial growth, Piazza and Schick have proved the vanishing of  $\eta_{\langle h\rangle}(\widetilde D)$ by a different method \cite[Theorem 13.7]{MR2294190}. 

In light of this algebraicity result, we propose the following question.

\begin{question*}
	What values can delocalized eta invariants take in general? Are they always algebraic numbers?
\end{question*}

In particular, if a delocalized eta invariant is transcendental, then it will lead to a counterexample to the Baum-Connes conjecture \cite{PBAC88, BCH94, AC94}. Note that the above question is a reminiscent of Atiyah's question concerning rationality of $\ell^2$-Betti numbers \cite{MR0420729}. Atiyah's question was answered in negative by Austin, who showed that $\ell^2$-Betti numbers can be transcendental \cite{MR3149852}. 

As another application of Theorem $\ref{thm:intro}$,  we give a $K$-theoretic proof of a version of the delocalized  Atiyah-Patodi-Singer index theorem. See Proposition $\ref{thm:aps}$ below. 

Some of the main results in this paper are inspired by previous work of Lott, \mbox{Leichtnam}, Piazza and Schick \cite{JLott92, MR1726745}\cite{MR2294190}\cite{ELPP01}. A key new ingredient of our approach is the construction of an explicit determinant map $\tau_h$  on $K_1(C_{L,0}^\ast(\widetilde M)^{\pi_1 M})$ for each non-identity conjugacy class $\langle h \rangle$ with polynomial growth. The definition of $K_1(C_{L,0}^\ast(\widetilde M)^{\pi_1 M})$ is reviewed in Section $\ref{sec:roelocal}$ below.  Each such determinant map is induced by the corresponding trace map $\tr_h$ on $K_0(C_r^\ast(\pi_1 M))$, and our construction is inspired by the work of de la Harpe and Skandalis \cite{MR743629} and  Keswani \cite{MR1763959}. In fact, combined with finite propagation speed of wave operators, our  $K$-theoretic approach above can also be used to give a uniform treatment of various vanishing results and homotopy invariance results for delocalized eta variants in \cite{MR952817, MR1763959, MR1794283, MR2294190, MR2761858, MR3296587}. See a brief discussion in Remark $\ref{rk:app}$ below.   We will  present the details in a separate paper \cite{tang-xie-yao-yu2}. 

One can use the same techniques developed in this paper to show that the analogues of  Theorem $\ref{thm:intro}$ and Proposition $\ref{thm:aps}$ hold for  hyperbolic groups,  if one uses Puschnigg's smooth dense subalgebras \cite{MR2647141}.  In fact, from the viewpoint of cyclic cohomology, the various traces considered in this paper are degree zero cyclic cocycles in the cyclic cohomology of the corresponding group algebra. For example, we can restate Theorem $\ref{thm:intro}$ by saying that Lott's delocalized eta invariants equal  the pairings between Higson-Roe's higher rho invariants and degree zero cyclic cocycles. The main results in this paper have natural analogues for higher degree cyclic cocycles. We shall apply the techniques from this paper to investigate the pairings between higher rho invariants and  cyclic cocycles of higher degrees in a sequel paper, where in particular we will show that analogues of the main results of this paper hold for higher degree cyclic cocycles, if $\pi_1 M$ has polynomial growth or $\pi_1 M$ is hyperbolic.

The paper is organized as follows. In Section $\ref{sec:conjpol}$, we recall some basic definitions of certain geometric $C^\ast$-algebras. Given a discrete group, for each conjugacy class with polynomial growth, we review how to extend a trace on the group algebra to the Connes-Moscovici smooth dense subalgebra of the corresponding reduced group $C^\ast$-algebra.  In Section $\ref{sec:second}$, we then use this extended trace map to define an explicit determinant map on secondary higher invariants. In Section $\ref{sec:deloc}$, we establish a precise connection between Higson-Roe's $K$-theoretic  higher rho invariants and Lott's delocalized eta invariants. We then apply it in Section $\ref{sec:bcalg}$ to prove an algebraicity result concerning the values of delocalized eta invariants and a version of delocalized Atiyah-Patodi-Singer index theorem. 

We would like to thank the referees for helpful comments. 
 
\section{Conjugacy classes with polynomial growth and extension of traces}\label{sec:conjpol}

Let $\Gamma$ be a discrete group and $\mathbb C\Gamma$ the corresponding group algebra. For each $h \in \Gamma$, there  is  a natural trace map on $\mathbb C\Gamma$ defined as follows: 
\[ \tr_h( a) = \sum_{g\in \langle h \rangle} a_g  \]
where $\langle h\rangle $ is the conjugacy class of $h$  and $a = \sum_{g\in \Gamma}a_g g\in \mathbb C\Gamma$. In this section, we give a brief construction on how to extend this trace to a smooth dense subalgebra of the reduced group $C^\ast$-algebra $C^\ast_r(\Gamma)$,  provided that  $\langle h\rangle $ has polynomial growth. In particular, such a trace map induces a map on $K_0(C_r^\ast(\Gamma))$. We shall use this induced map on $K_0(C_r^\ast(\Gamma))$ to define a determinant map on secondary higher invariants in the next section.

\subsection{Roe algebras and localization algebras}\label{sec:roelocal}
In this subsection, we briefly recall some standard definitions of certain geometric $C^\ast$-algebras. We refer the reader to \cite{MR1147350, MR1451759} for more details. Let $X$ be a proper metric space. That is, every closed ball in $X$ is compact. An	$X$-module is a separable Hilbert space equipped with a	$\ast$-representation of $C_0(X)$, the algebra of all continuous functions on $X$ which vanish at infinity. An	$X$-module is called nondegenerate if the $\ast$-representation of $C_0(X)$ is nondegenerate. An $X$-module is said to be standard if no nonzero function in $C_0(X)$ acts as a compact operator. 
\begin{definition}
	Let $H_X$ be a $X$-module and $T$ a bounded linear operator acting on $H_X$. 
	\begin{enumerate}[(i)]
		\item The propagation of $T$ is defined to be $\sup\{ d(x, y)\mid (x, y)\in \supp(T)\}$, where $\supp(T)$ is  the complement (in $X\times X$) of the set of points $(x, y)\in X\times X$ for which there exist $f, g\in C_0(X)$ such that $gTf= 0$ and $f(x)\neq 0$, $g(y) \neq 0$;
		\item $T$ is said to be locally compact if $fT$ and $Tf$ are compact for all $f\in C_0(X)$; 
		\item $T$ is said to be pseudo-local if $[T, f]$ is compact for all $f\in C_0(X)$.  
	\end{enumerate}
\end{definition}

\begin{definition}
	Let $H_X$ be a standard nondegenerate $X$-module and $\mathcal B(H_X)$ the set of all bounded linear operators on $H_X$.  
	\begin{enumerate}[(i)]
		\item The Roe algebra of $X$, denoted by $C^\ast(X)$, is the $C^\ast$-algebra generated by all locally compact operators with finite propagations in $\mathcal B(H_X)$.
		\item $D^\ast(X)$ is the $C^\ast$-algebra generated by all pseudo-local operators with finite propagations in $\mathcal B(H_X)$. In particular, $D^\ast(X)$ is a subalgebra of the multiplier algebra of $C^\ast(X)$.
		\item $C_L^\ast(X)$ (resp. $D_{L}^\ast(X)$) is the $C^\ast$-algebra generated by all bounded and uniformly norm-continuous functions $f: [0, \infty) \to C^\ast(X)$ (resp. $f: [0, \infty) \to D^\ast(X)$)  such that 
		\[ \textup{propagation of $f(t) \to 0 $, as $t\to \infty$. }\]  
		Again $D_{L}^\ast(X)$ is a subalgebra of the multiplier algebra of $C_L^\ast(X)$. 
		\item $C_{L, 0}^\ast(X)$ is the kernel of the evaluation map 
		\[  \ev : C_L^\ast(X) \to C^\ast(X),  \quad   \ev (f) = f(0).\]
		In particular, $C_{L, 0}^\ast(X)$ is an  ideal of $C_L^\ast(X)$. Similarly, we define $D_{L, 0}^\ast(X)$ as the kernel of the evaluation map from $D_L^\ast(X)$ to $D^\ast(X)$.
		
	\end{enumerate}
\end{definition}

Now in addition we assume that a discrete group $\Gamma$ acts properly and cocompactly on  $X$  by isometries. Let $H_X$ be a $X$-module equipped with a covariant unitary representation of $\Gamma$. If we denote the representation of $C_0(X)$ by $\varphi$ and the representation of $\Gamma$ by $\pi$, this means 
\[  \pi(\gamma) (\varphi(f) v )  =  \varphi(f^\gamma) (\pi(\gamma) v),\] 
where $f\in C_0(X)$, $\gamma\in \Gamma$, $v\in H_X$ and $f^\gamma (x) = f (\gamma^{-1} x)$. In this case, we call $(H_X, \Gamma, \varphi)$ a covariant system.  

\begin{definition}[\cite{MR2732068}]
	A covariant system $(H_X, \Gamma, \varphi)$ is called admissible if 
	\begin{enumerate}[(1)]
		\item the $\Gamma$-action on $X$ is proper and cocompact;
		\item $H_X$ is a nondegenerate standard $X$-module;
		\item for each $x\in X$, the stabilizer group $\Gamma_x$ acts on $H_X$ regularly in the sense that the action is isomorphic to the action of $\Gamma_x$ on $l^2(\Gamma_x)\otimes H$ for some infinite dimensional Hilbert space $H$. Here $\Gamma_x$ acts on $l^2(\Gamma_x)$ by translations and acts on $H$ trivially. 
	\end{enumerate}
\end{definition}
We remark that for each locally compact metric space $X$ with a proper and cocompact isometric action of $\Gamma$, there exists an admissible covariant system $(H_X, \Gamma, \varphi)$. Also, we point out that the condition $(3)$ above is automatically satisfied if $\Gamma$ acts freely on $X$. If no confusion arises, we will denote an admissible covariant system $(H_X, \Gamma, \varphi)$ by $H_X$ and call it an admissible $(X, \Gamma)$-module.

\begin{definition}
	Let $X$ be a locally compact metric space $X$ with a proper and cocompact isometric action of $\Gamma$. If $H_X$ is an admissible $(X, \Gamma)$-module, we denote by $\mathbb C[X]^\Gamma$ the $\ast$-algebra of all $\Gamma$-invariant locally compact operators with finite propagations in $\mathcal B(H_{X})$.  We define $C^\ast(X)^\Gamma$ to be the completion of $\mathbb C[X]^\Gamma$ in $\mathcal B(H_{ X})$.
\end{definition}
Since the action of $\Gamma$ on $X$ is cocompact,  we have $C^\ast(X)^\Gamma \cong C^\ast_r(\Gamma)\otimes \mathcal K$, where $ C^\ast_r(\Gamma)$ is the reduced group $C^\ast$-algebra of $\Gamma$ and $\mathcal K$ is the algebra of all compact operators, cf. \cite[Lemma 5.14]{MR1399087}.

Similarly, we can also define  $D^\ast( X)^\Gamma$, $C_L^\ast( X)^\Gamma$, $D_L^\ast( X)^\Gamma$,  $C_{L,0}^\ast( X)^\Gamma$, $D_{L, 0}^\ast( X)^\Gamma$, $C^\ast_{L}( Y;  X)^\Gamma$ and $C^\ast_{L,0}( Y;  X)^\Gamma$.   

\begin{remark}
	Up to isomorphism, $C^\ast(X) = C^\ast(X, H_X)$ does not depend on the choice of the standard nondegenerate $X$-module $H_X$. The same holds for $D^\ast(X)$, $C_L^\ast(X)$, $D_L^\ast( X)$,  $C_{L,0}^\ast( X)$, $D_{L, 0}^\ast(X)$, $C^\ast_{L}( Y;X)$, $C^\ast_{L,0}( Y;X)$ and their $\Gamma$-equivariant versions. 
\end{remark}

\begin{remark}
	Note that we can also define maximal versions of all the $C^\ast$-algebras above. For example, we define the maximal $\Gamma$-invariant Roe algebra $C^\ast_{\max}(X)^\Gamma$ to be the completion of $\mathbb C[X]^\Gamma$ under the maximal norm:
	\[  \|a\|_{\max} = \sup_{\phi} \ \big\{\|\phi(a)\| \mid \phi: \mathbb C[X]^\Gamma \to \mathcal B(H') \ \textup{a $\ast$-representation} \big\}. \]
\end{remark}

\subsection{Extension of traces to smooth dense subalgebras}\label{sec:exttr}
In this subsection, we give a brief construction on how to extend the trace $\tr_h\colon \mathbb C\Gamma \to \mathbb C$ to be defined on a smooth dense subalgebra of the reduced group $C^\ast$-algebra $C^\ast_r(\Gamma)$,  provided that  $\langle h\rangle $ has polynomial growth. Also see \cite{Samurkas:2017aa} for an alternative approach and relevant applications.

Let $M$ be a closed oriented Riemannian manifold. 
Let $\smooth$ be the algebra of smoothing operators on $M$. Fix a basis of $L^2(M)$, then $\smooth$ can be identified with the algebra of matrices $(a_{ij})_{i, j\in \mathbb N}$ such that 
\[  \sup_{i, j}i^kj^l|a_{ij}| < \infty \textup{ for all } k, l\in \mathbb N. \]

Let us recall the following smooth dense subalgebra of $C_r^\ast(\Gamma)\otimes \mathcal K$, due to  Connes and Moscovici \cite{CM90}. Let $\Delta$, resp. $D$, be the unbounded operator in $\ell^2(\mathbb N)$, resp. $\ell^2(\Gamma)$, defined by 
\[ \Delta(\delta_j) = j \delta_j \textup{ for } j \in \mathbb N, \textup{ resp. } Dg = |g|\cdot g \textup{ for } g\in \Gamma.   \]
Consider the unbounded derivations $\partial = [D, \cdot]$ of $\mathcal B(\ell^2(\Gamma))$ and $\widetilde \partial = [D\otimes I, \cdot]$ of $\mathcal B(\ell^2(\Gamma)\otimes \ell^2(\mathbb N))$, and set\footnote{To be precise, the algebra $\mathscr B(\widetilde M)^\Gamma$ defined here is slightly different from Connes-Moscovici's algebra $\mathscr B$ in \cite[Lemma 6.4]{CM90}. Both of them are smooth dense subalgebras of $C_r^\ast(\Gamma)\otimes \mathcal K$. In this paper, the algebra $\mathscr B(\widetilde M)^\Gamma$ works better for our purposes.   }:
\[ \mathscr B(\widetilde M)^\Gamma= \{ A\in C_r^\ast(\Gamma)\otimes \mathcal K \mid \widetilde \partial^k(A)\circ (I\otimes \Delta)^2 \textup{ is bounded } \forall k \in \mathbb N \}.   \]
It follows from \cite[Lemma 6.4]{CM90} (and its proof) that $\mathscr B(\widetilde M)^\Gamma$ contains $\mathbb C\Gamma\otimes \mathscr R$ and is closed under holomorphic functional calculus. 

For each $n\in \mathbb N$, define the following seminorm on $\mathscr B(\widetilde M)^\Gamma$:
\[ \|A\|_n = \sum_{k = 0}^n \frac{1}{k!} \big\|\widetilde \partial^k(A)\circ (I\otimes \Delta)^2\big\|_{op}  \]
where $\|\widetilde \partial^k(A)\circ (I\otimes \Delta)^2\big\|_{op}$ stands for the operator norm of $\widetilde \partial^k(A)\circ (I\otimes \Delta)^2$.  Then $\mathscr B(\widetilde M)^\Gamma$ is a Fr\'echet algebra under this sequence of seminorms $\{\|\cdot \|_n \colon n\in \mathbb N\}$.

We associate to each element $h\in \Gamma$ the following trace on $\mathbb C\Gamma\otimes \mathscr R$: 
\[ \tr_h( \gamma \otimes \omega) =  \begin{cases}
\textup{trace}(\omega) & \textup{if } \gamma \in \langle h\rangle,  \\   0 & \textup{if } \gamma \notin \langle h \rangle, 
\end{cases} \]
for $\gamma \in \Gamma $ and $\omega \in \mathscr R$.  Now suppose $h\in \Gamma$ such that its conjugacy class $\langle h\rangle $ has polynomial growth, that is, there exists $C$ and $d$ such that 
\[  \sharp \{g\in \langle h\rangle :  |g|\leq n\} \leq C\cdot n^d. \]
The following lemma shows that the trace $\tr_h$ extends to a continuous trace on $\mathscr B(\widetilde M)^\Gamma$, provided that $\langle h\rangle $ has polynomial growth. 
\begin{lemma}\label{lm:tr}
If $\langle h\rangle $ has polynomial growth, then	$\tr_h$ extends to a continuous trace on  $\mathscr B(\widetilde M)^\Gamma$.  
\end{lemma}
\begin{proof}
Our proof will follow closely the proof of \cite[Lemma 6.4]{CM90}. 	If $A\in \mathscr B(\widetilde M)^\Gamma$, $A = (a_{ij})_{i, j\in \mathbb N}$ with $a_{ij}\in C_r^\ast(\Gamma)$, we define 
\begin{equation}\label{eq:sum}
\tr_h (A)  = \sum_{j\in \mathbb N}\sum_{g\in \langle h \rangle} a_{jj}(g).
\end{equation} 
We need to verify that the summation on the right side converges. Consider the following inequality: 
\begin{align*}
\sum_{j\in \mathbb N}\sum_{g\in \langle h \rangle} |a_{jj}(g)| \leq \Big(\sum_{j\in \mathbb N}\sum_{g\in \langle h \rangle} j^2 (1 + |g|)^{2k}|a_{jj}(g)|^2\Big)^{1/2}  \Big(\sum_{j \in \mathbb N} \sum_{g\in \langle h \rangle} j^{-2}(1 + |g|)^{-2k}  \Big)^{1/2} 
\end{align*}
Since $\langle h \rangle$ has polynomial growth, the term $\sum_{j \in \mathbb N} \sum_{g\in \langle h \rangle} j^{-2}(1 + |g|)^{-2k} $ converges  by choosing a sufficiently large $k$, for example, $k > (d+1)/2$. 

On the other hand, observe that 
\begin{align*}
\left(\widetilde \partial^k(A)\circ (I\otimes \Delta)^2\right) (\delta_e \otimes \delta_j) &= j^2 \sum_{i\in \mathbb N}\widetilde \partial^k a_{ij}(\delta_e)\otimes \delta_i \\
& =  j^2 \sum_{(g, i)\in \Gamma \times \mathbb N} |g|^k a_{ij}(g) \delta_g \otimes \delta_i,
\end{align*} 
which implies
\[ \sum_{(g, i)\in \Gamma \times \mathbb N} j^2 |g|^{2k} |a_{ij}(g)|^2 \leq  j^{-2}  \big\|\widetilde \partial^k(A)\circ (I\otimes \Delta)^2\big\|^2_{op}.   \]
It follows that there exists a fixed constant $C_k >0$ such that 
\[  \sum_{j\in \mathbb N}\sum_{g\in \langle h \rangle} j^2 (1 + |g|)^{2k}|a_{jj}(g)|^2 \leq C_k \cdot \|A\|_k^2. \]
 Therefore, $\tr_h$ extends to a continuous linear map on  $\mathscr B(\widetilde M)^\Gamma$.  

Now let us verify that $\tr_h$ is a trace on $\mathscr B(\widetilde M)^\Gamma$, that is, $\tr_h(AB) = \tr_h(BA)$ for $A, B\in \mathscr B(\widetilde M)^\Gamma$. First, assume that $A, B\in \mathbb C\Gamma\otimes \mathscr R$.  In this case, a straightforward calculation shows that    $\tr_h(AB) = \tr_h(BA)$. Now the general case follows, since $\tr_h$ is continuous and  $\mathbb C\Gamma\otimes \mathscr R$ is dense in $\mathscr B(\widetilde M)^\Gamma$. This finishes the proof. 
\end{proof}

Since $\mathscr B(\widetilde M)^\Gamma$ is a dense subalgebra of $C^\ast(\widetilde M)^\Gamma \cong C_r^\ast(\Gamma)\otimes \mathcal K$ and is closed under holomorphic functional calculus, we see that the trace $\tr_h$ induces a homomorphism: 
\[ \tr_h\colon  K_0(\mathscr B(\widetilde M)^\Gamma) =  K_0(C_r^\ast(\Gamma)\otimes \mathcal K) \to \mathbb C.\]

\section{Secondary higher invariants and determinant maps}\label{sec:second}

In this section, we will use the trace maps from the previous section to construct certain determinant maps on secondary higher invariants. More precisely,  for each $h\neq e\in \Gamma$ and 
\[ \tr_h\colon  K_0(\mathscr B(\widetilde M)^\Gamma) =  K_0(C_r^\ast(\Gamma)\otimes \mathcal K) \to \mathbb C,\]
we will construct a linear map $\tau_h\colon K_1(C_{L,0}^\ast(\widetilde M)^\Gamma) \to \mathbb C$. Let us first introduce some notation. 
\begin{definition}
	We define $\mathscr B_L(\widetilde M )^\Gamma$ to be the dense subalgebra of $C_L^\ast(\widetilde M)^\Gamma$ consisting of elements $f\in C_L^\ast(\widetilde M)^\Gamma$ such that   $f(t) \in \mathscr B(\widetilde M)^\Gamma$ for all $t\in [0, \infty)$ and $f$ is piecewise smooth with respect to the Fr\'echet topology of $\mathscr B(\widetilde M)^\Gamma$.
\end{definition}

$\mathscr B_L(\widetilde M )^\Gamma$ is a dense subalgebra of $C_L^\ast(\widetilde M)^\Gamma$ and is closed under holomorphic functional calculus. Similarly, we define $\mathscr B_{L, 0}(\widetilde M )^\Gamma$ to be the kernel of the evaluation map 
\[ \ev\colon  \mathscr B_L(\widetilde M )^\Gamma \to  \mathscr B(\widetilde M)^\Gamma \textup{ defined by } f\mapsto f(0).\]
The following lemma is an immediate consequence of the above definitions. 

\begin{lemma}\label{lm:sm}
The inclusion maps $\mathscr B_L(\widetilde M )^\Gamma \hookrightarrow  C_L^\ast(\widetilde M)^\Gamma$ and $\mathscr B_{L, 0}(\widetilde M )^\Gamma \hookrightarrow  C_{L, 0}^\ast(\widetilde M)^\Gamma$ induce natural isomorphisms: 
\[ K_j(\mathscr B_L(\widetilde M )^\Gamma) \cong K_j(C_L^\ast(\widetilde M)^\Gamma) \textup{ and } K_j(\mathscr B_{L, 0}(\widetilde M )^\Gamma) \cong K_j(C_{L, 0}^\ast(\widetilde M)^\Gamma).\]  
\end{lemma}

Now for each non-identity element $h\in \Gamma$ such that the conjugacy class $\langle h\rangle$ has polynomial growth,  we shall construct a determinant map 
\[ \tau_h\colon K_1(\mathscr B_{L,0}(\widetilde M)^\Gamma) \to \mathbb C, \]
which can be equivalently viewed as a map $ \tau_h\colon K_1(C_{L,0}^\ast(\widetilde M)^\Gamma) \to \mathbb C$. Roughly speaking, the explicit formula for $\tau_h$ is given by 
\begin{equation}\label{eq:det}
	\tau_h(u) \coloneqq \frac{1}{2\pi i}\int_{0}^\infty \tr_h\big(\dot{u}(t)u^{-1}(t)\big) dt
\end{equation}
for each $[u]\in K_1(\mathscr B_{L,0}(\widetilde M)^\Gamma)$, 	where $\dot u$ is the derivative of $u$. 
In order to justify the validity of this integral,   we need the following technical results.

\begin{definition}\label{def:locloop}
Let $SC^\ast(\widetilde M)^\Gamma$ be the suspension of $C^\ast(\widetilde M)^\Gamma$, and  $\varphi\in$ be an invertible element in $ SC^\ast(\widetilde M)^\Gamma$, that is,  a loop $\varphi\colon S^1 = [0, 1]/\{0, 1\} \to (C^\ast(\widetilde M)^\Gamma)^+$ of invertible elements such that $\varphi(1) =1$, where  $(C^\ast(\widetilde M)^\Gamma)^+$ is the unitization $C^\ast(\widetilde M)^\Gamma$. We say  $\varphi$ is \emph{local} if it is the image of an invertible element $\psi\in SC^\ast_{L}(\widetilde{M})^\Gamma$ under the evaluation map $SC^\ast_{L}(\widetilde{M})^\Gamma \to SC^\ast(\widetilde{M})^\Gamma$. Similarly, an invertible element $\varphi \in S\mathscr B(\widetilde M)^\Gamma$ is called local if it is the image of an invertible element  $\psi\in S\mathscr B_{L}(\widetilde{M})^\Gamma$ under the evaluation map. 
\end{definition}

Local loops of invertible elements have the following property. 

\begin{lemma}\label{lm:local}
If $\varphi$ is a local invertible element in  $SC^\ast(\widetilde M)^\Gamma$ \textup{(}resp. $S\mathscr B(\widetilde M)^\Gamma$\textup{)}, then for $\forall \varepsilon >0$, there exists an idempotent  $p$ in $C^\ast(\widetilde M)^\Gamma$ \textup{(}resp. $\mathscr B(\widetilde M)^\Gamma$\textup{)} such that   the propagation of $p$ is $\leq \varepsilon$  and $\varphi$ is equivalent to the invertible  element $e^{2\pi i\theta} p + (1-p)$ in $SC^\ast(\widetilde M)^\Gamma$ \textup{(}resp. $S\mathscr B(\widetilde M)^\Gamma$\textup{)}.
\end{lemma}
\begin{proof}	
	By the Bott periodicity map
	\[  \beta\colon K_0(C_L^\ast(\widetilde M)^\Gamma) \xrightarrow{\ \cong \ } K_1(SC_L^\ast(\widetilde M)^\Gamma),\quad  P \mapsto e^{2\pi i \theta}P + (1-P),     \]
every invertible element in $SC_L^\ast(\widetilde X)^\Gamma$ is equivalent to  an invertible element of the form $e^{2\pi i \theta}P + (1-P)$. It follows from the Baum-Douglas geometric description of $K$-homology \cite[Section 11]{BD82} that $P$ can be chosen to be a family of idempotents such that the propagation of $P(t)$ goes to zero as $t$ goes to infinity.  Indeed, since $P$ represents a $K$-homology class, it can be chosen to be the local index (cf. \cite[Section 3]{MR1451759}) of a twisted Dirac operator over a $spin^c$ manifold. The standard construction of $K$-theoretic (local) index classes shows that the propagation of the idempotent\footnote{It is important that we use idempotents instead of projections here. In general, the construction of $K$-theoretic index classes does \emph{not} produce a \emph{projection} (an idempotent that is self-adjoint) with finite propagation, but it does produce an idempotent with arbitrary small propagation. On the other hand, if one insists on having both self-adjointness and finite propagation, one possibility is to  use quasi-projections, cf.  \cite{GY98}.  } $P(t)$ can be made finite and goes to zero as $t$ goes to infinity. 

Since $\varphi$ is the image of an invertible element in $SC_L^\ast(\widetilde X)^\Gamma$ under the evaluation map, it follows immediately that for $\forall \varepsilon >0$, there exists an idempotent $p \in C^\ast(\widetilde M)^\Gamma$ such that the propagation of $p$ is $\leq \varepsilon$ and  $\varphi$ is homotopic to the loop $e^{2\pi i \theta} p + (1-p)$ through a family of loops of invertible elements.  

By applying Lemma $\ref{lm:sm}$, the case of a local invertible element in $S\mathscr B(\widetilde M)^\Gamma$ also follows.
\end{proof}

In order to rigorously define the determinant map $\tau_h\colon K_1(\mathscr B_{L,0}(\widetilde M)^\Gamma) \to \mathbb C$, we shall prove that every element of $K_1(\mathscr B_{L,0}(\widetilde M)^\Gamma)$ has a nice representative with certain regularities. The main motivation for choosing such nice representatives is to guarantee the convergence of the integral in line $\eqref{eq:det}$.   Moreover, we show that  for a given element of $K_1(\mathscr B_{L,0}(\widetilde M)^\Gamma)$, two different such regularized representatives can be connected by a family of representatives of the same kind. This allows us to show that the integral in line  $\eqref{eq:det}$ is independent of the choice of such representatives.  

\begin{proposition}\label{prop:reg}
Every element $[u]\in K_1(\mathscr B_{L,0}(\widetilde M)^\Gamma)$  has a representative $w\colon [0, \infty) \to (\mathscr B(\widetilde M)^\Gamma)^+$ such that 
\[ w(t) = \begin{cases}
u(t) & \textup{ if $0\leq t \leq 1$,}\\
h(t) & \textup{ if $1\leq t \leq 2$,} \\
e^{2\pi i \frac{F(t-1) + 1}{2}} & \textup{ if $t \geq 2$,} 
\end{cases} \]
where $h$  is a piecewise smooth path of invertible elements connecting $u(1)$ and $e^{2\pi i \frac{F(1) + 1}{2}}$, and  $F$ is a piecewise smooth map $F\colon [1, \infty) \to  D^\ast(\widetilde M)^\Gamma$ satisfying 
\begin{enumerate}[$(1)$]
	\item $F(t)^2 - 1 \in \mathscr B(\widetilde M)^\Gamma $ and $F^\ast(t) = F(t)$, 
	\item its derivative $F'(t)\in \mathscr B(\widetilde M)^\Gamma,$ 
	\item and propagation of $F(t)$ goes to $0$, as $t\to \infty$.  
\end{enumerate}
Moreover, if $v$ is another such representative,  then there exists a piecewise smooth family  of invertibles $u_s\in \mathscr B_{L,0}(\widetilde M)^\Gamma$ and piecewise smooth maps $F_s\colon [1, \infty) \to  D^\ast(\widetilde M)^\Gamma$ satisfying conditions $(1), (2)$ and $(3)$ above,  with $s\in [0, 1]$,  such that
\begin{enumerate}[$(i)$]
	\item $u_0 = w$, 
	\item $u_s(t) = e^{2\pi i \frac{F_s(t-1) + 1}{2}} \textup{ for all } t\geq 2; $
	\item $u_1 v^{-1}(t) = 1 $ for all $t\notin (1, 2)$ and $u_1v^{-1}\colon [1, 2]\to \mathscr B(\widetilde M)^\Gamma$  is a local loop of invertible elements.  
\end{enumerate} 

\end{proposition}
\begin{remark}
	For simplicity, we shall call a representative as in the proposition a \emph{regularized} representative. 
\end{remark}
\begin{proof}
	View the invertible element $u \in \mathscr B_{L,0}(\widetilde M)^\Gamma$ as an invertible element in $\mathscr B_{L}(\widetilde M)^\Gamma$. Consider the element\footnote{ Note that $\hat u$ starts at $t=1$ instead of $t=0$.} $\hat u = u \colon [1, \infty) \to (\mathscr B(\widetilde M)^\Gamma)^+$ in $K_1(\mathscr B_{L}(\widetilde M)^\Gamma)$. Since the $K$-theory of $\mathscr B_{L}(\widetilde M)^\Gamma$ is the $K$-homology of $M$, it follows from the Baum-Douglas geometric description of $K$-homology \cite[Section 11]{BD82} that $\hat u$ can be represented by a twisted Dirac operator over a $spin^c$ manifold. In particular, it follows that there exists a piecewise smooth map $F\colon [1, \infty) \to  D^\ast(\widetilde M)^\Gamma$ satisfying 
	\begin{enumerate}[$(1)$]
		\item $F(t)^2 - 1 \in \mathscr B(\widetilde M)^\Gamma$ and $F^\ast(t) = F(t)$,
		\item its derivative $F'(t)\in \mathscr B(\widetilde M)^\Gamma,$ 
		\item propagation of $F(t)$ goes to $0$, as $t\to \infty$; 
	\end{enumerate}
and $\hat u(t)$ is homotopic to the path  $ e^{2\pi i \frac{F(t) + 1}{2}}$ with $t\in [1, \infty). $ In particular, there is a path of invertible elements, denoted by $h$,  connecting $u(1)$ and $e^{2\pi i \frac{F(1) + 1}{2}}$. 
Then $u$ is homotopic to the invertible element $w$ defined by 
\[ w(t) = \begin{cases}
u(t) & \textup{ if $0\leq t \leq 1$,}\\
h(t) & \textup{ if $1\leq t \leq 2$,} \\
 e^{2\pi i \frac{F(t-1) + 1}{2}} & \textup{ if $t \geq 2$,} 
\end{cases} \]
cf. Figure $\ref{fig:homtpy}$ below. 

\begin{figure}[h]
	\begin{center}
		\begin{tikzpicture}
    \draw [fill= lightgray, lightgray] (0,0) rectangle (10,-2);
		\draw [thick] (0,0) -- (2,0) -- (4, 0) -- (10, 0);
		
		\filldraw [black] (2,0) circle (1pt);
		\filldraw [black] (4,0) circle (1pt);  
		 \node [above] at (1,0) {$u(t)$};
		 \node [above] at (3,0) {$h(t)$};
		 \node [above] at (6,0) {$e^{2\pi i \frac{F(t-1) + 1}{2}}$};

		\draw [thick] (0,-2) -- (2,-2) -- (4, -2) -- (10, -2);
		\filldraw [black] (2,-2) circle (1pt);
		\draw (2, -2) -- (4, 0); 
		\node [left] at (2.9, -1) {$h(t)$}; 
		
		 \node [below] at (1,-2) {$u(t)$};
		\node [below] at (5,-2) {$u(t)$};
		\end{tikzpicture}
	\end{center}
	\caption{homotopy between $u$ and $w$.} \label{fig:homtpy}
\end{figure}
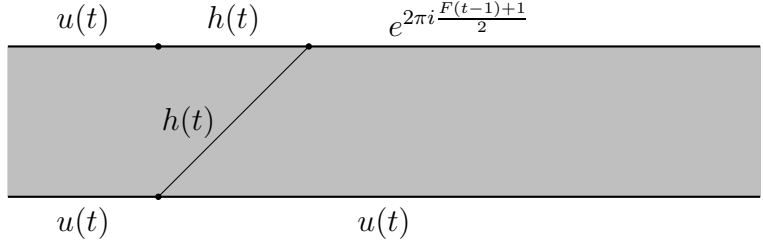

Now suppose $v$ is another representative of $[u]$ such that 
\[ v(t) = \begin{cases}
u(t) & \textup{ if $0\leq t \leq 1$,}\\
g(t) & \textup{ if $1\leq t \leq 2$,} \\
e^{2\pi i \frac{G(t-1) + 1}{2}} & \textup{ if $t \geq 2$,} 
\end{cases} \]
where $g$ is a path of invertible elements connecting $u(1)$ and $e^{2\pi i \frac{G(1) + 1}{2}}$, and $G$ is a piecewise smooth map $G\colon [1, \infty) \to  D^\ast(\widetilde M)^\Gamma$ satisfying that  $G(t)^2 - 1 \in \mathscr B(\widetilde M)^\Gamma $ and $G^\ast(t) = G(t)$; its derivative $G'(t)\in \mathscr B(\widetilde M)^\Gamma$;  and propagation of $G(t)$ goes to $0$, as $t\to \infty$. 

By \cite[Theorem 3.8]{MR1701826}, there exists a piecewise smooth family $F_s\colon [1, \infty) \to  D^\ast(\widetilde M)^\Gamma$ with $s\in [0, 1]$  such that $F_0 = F$ and $F_1 = G$;   $F_s(t)^2 - 1 \in \mathscr B(\widetilde M)^\Gamma $ and $F_s^\ast(t) = F_s(t)$; its derivative $\frac{\partial}{\partial t}F_s(t)\in \mathscr B(\widetilde M)^\Gamma;$ and propagation of $F_s(t)$ goes to $0$, as $t\to \infty$.  

Let $\varpi \colon [0, \infty) \to (\mathscr B(\widetilde M)^\Gamma)^+$ be the path of invertibles defined as 
\[ \varpi (t) = \begin{cases}
u(t) & \textup{ if $0\leq t \leq 1$,}\\
h(t) & \textup{ if $1\leq t \leq 2$,}\\
e^{2\pi i \frac{F_{s}(1) + 1}{2}} & \textup{ if $2\leq t  = s+2 \leq 3$,} \\
e^{2\pi i \frac{G(t-2) + 1}{2}} & \textup{ if $t \geq 3$,} 
\end{cases} \]
Clearly,  $w$ is homotopic to $\varpi $. On the other hand, after a re-parametrization, it is not difficult to see that $\varpi$ differs from $v$ by the loop $f\colon [0, 1]\to  (\mathscr B(\widetilde M)^\Gamma)^+$ with 
\[  f(t) = \begin{cases}
g(t)^{-1}h(2t) & \textup{if $0 \leq t \leq 1/2$}\\
\\
g(t)^{-1}e^{2\pi i \frac{F_{2t-1}(1) + 1}{2}}  & \textup{ if $1/2\leq t \leq 1$.} 
\end{cases} \]
Moreover,  $f$ is a local loop in the sense of Definition $\ref{def:locloop}$ (cf. Figure $\ref{fig:local}$). This finishes the proof.
\begin{figure}[h]
	\begin{center}
		\begin{tikzpicture}
		
		\draw [thick](0,0) -- (2,0) -- (4, 0) -- (15, 0);
		\filldraw [black] (2,0) circle (1.5pt);
		\filldraw [black] (4,0) circle (1.5pt);  
		\draw [thick] (0,-2) -- (2,-2) -- (3.5, -2);
		\draw [thick, dashed] (3.5, -2) -- (4.5, -2);
		\filldraw [black] (2,-2) circle (1.5pt);
		
		 \node [above] at (1,0) {$u(t)$};
		\node [above] at (3,0) {$h(t)$};
		\node [above] at (6,0) {$e^{2\pi i \frac{F(t-1) + 1}{2}}$};

		\draw [thick] (0,-4) -- (2,-4) -- (4, -4) -- (15, -4);
		\filldraw [black] (2,-4) circle (1.5pt);
		\filldraw [black] (4,-4) circle (1.5pt);
		
		\node [left] at (2.9, -1) {$h(t)$}; 
		\node [left] at (2.9, -3) {$g(t)$}; 
		\node [right] at (4.6, -1.6) { $e^{2\pi i \frac{F_s(1) + 1}{2}}$ };
				
		\draw (4, 0) --  (2, -2);
		\draw (2, -2) -- (4, -4);
		\draw (4, 0) to [out = -60, in= 60] (4, -4);
		
		\draw [fill= lightgray] (4,0) -- (14, 0)  to [out = -60, in= 60] (14, -4) -- (4, -4)  to [out = 60, in= -60] (4, 0);
		
		\draw[dashed]  (14, -4) --  (12, -2) -- (14, 0) ;
		\draw (14, 0) to [out = -60, in= 60] (14, -4);
		
		\filldraw [black] (14,0) circle (1.5pt);
		\filldraw [black] (14,-4) circle (1.5pt);
		\filldraw [black] (12,-2) circle (1.5pt);
		
		 \node [below] at (1,-4) {$u(t)$};
		\node [below] at (3,-4) {$g(t)$};
		\node [below] at (6,-4) {$e^{2\pi i \frac{G(t-1) + 1}{2}}$};
		\end{tikzpicture}
	\end{center}
\caption{$\varpi$ and $v$ differs by a local loop. The picture should be viewed as 3-dimensional. The circular sectors form an element in $S\mathscr B_L(\widetilde M)^\Gamma$, which shows that the circular sector on the left is local in the sense of the Definition $\ref{def:locloop}$.} \label{fig:local}
\end{figure}
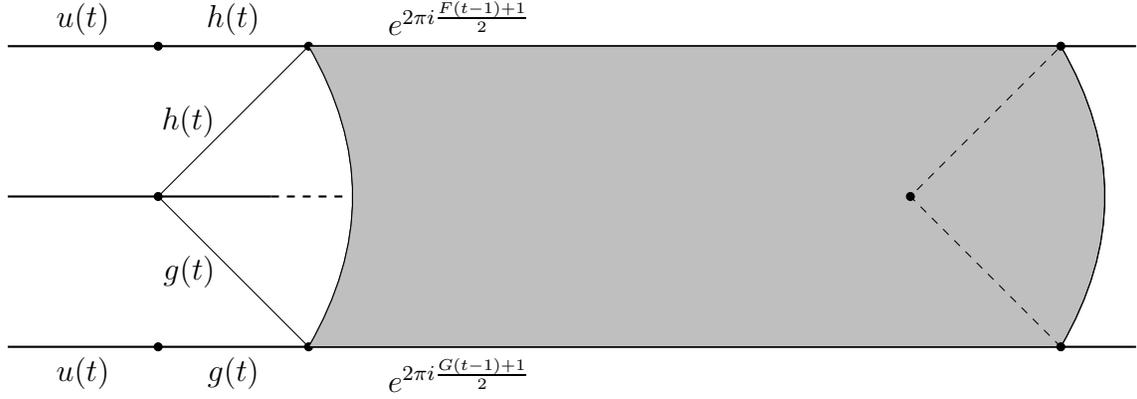

\end{proof}

As before, suppose that $h$ is a non-identity element in $\Gamma$ such that its conjugacy class $\langle h\rangle$ has polynomial growth.

\begin{definition}\label{def:reg}
	For each $[u]\in K_1(\mathscr B_{L,0}(\widetilde M)^\Gamma)$, let $w$ be a regularized representative of $u$ as in Proposition $\ref{prop:reg}$.  We define
	\begin{equation}\label{eq:int}
	\tau_h(u) \coloneqq \frac{1}{2\pi i}\int_{0}^\infty \tr_h\big(\dot{w}(t)w^{-1}(t)\big) dt.
	\end{equation} 
where $\dot w$ is the derivative of $w$. 
\end{definition}

Let us show that the above formula $\eqref{eq:int}$ gives a well-defined map $\tau_h\colon K_1(\mathscr B_{L,0}(\widetilde M)^\Gamma) \to \mathbb C$. In particular, we shall prove that the integral in the formula $\eqref{eq:int}$ converges and  is independent of the choice of regularized representative. 

\begin{proposition} \label{prop:det} If  the conjugacy class $\langle h\rangle$ of a non-identity element $h\in \Gamma $ has polynomial growth, then 
	the map $\tau_h \colon K_1( C^\ast_{L,0}(\widetilde M)^\Gamma) \to \mathbb C$ is well-defined. 
\end{proposition}
\begin{proof}
 	Let $[u]\in K_1( C^\ast_{L,0}(\widetilde M)^\Gamma) \cong K_1(\mathscr B_{L,0}(\widetilde M)^\Gamma) $.  Let $w$ be a regularized representative of $[u]$ as in Proposition $\ref{prop:reg}$.  First, we shall show that the integral in the formula $\eqref{eq:int}$ converges. Indeed, we have 
	\begin{align*}
	\frac{1}{2\pi i} \int_{0}^\infty \tr_h\big(\dot{w}(t)w^{-1}(t)\big) dt  = \frac{1}{2\pi i} \int_{0}^2  \tr_h\big(\dot{w}(t)w^{-1}(t)\big) dt  + \frac{1}{2\pi i} \int_{2}^\infty \tr_h(\dot F(t)) dt \\
	\end{align*}
	The first integral on the right is clearly well-defined. Observe that there exists $\varepsilon >0$ (depending only on $\widetilde M$) such that $\tr_h(\dot F(t)) = 0$ as long as the propagation of $\dot F(t)$ is less than $\varepsilon$.  Since the propagation of $\dot F(t)$ goes $0$  as $t$ goes to $\infty$,  it follows that the second integral on the right is well-defined.

	 Now let us show that $\tau_h([u])$ is independent of the choice of regularized representatives. suppose $v$ is another regularized representative of $[u]$. By Proposition $\ref{prop:reg}$, there exists a piecewise smooth family  of invertibles $u_s\in \mathscr B_{L,0}(\widetilde M)^\Gamma)^+$ with the stated properties $(i) - (iii)$ as in Proposition $\ref{prop:reg}$. A key consequence of these properties is that it guarantees the convergence of each integral in the following transgression formula:
	\begin{align}
	\frac{\partial}{\partial s} \tau_h(u_s) & = \dashint_{0}^\infty  \partial_s \tr_h\big((\partial_t u)u^{-1}\big) dt \label{eq:trans} \\
	& = \dashint_{0}^\infty  \tr_h\big((\partial_s\partial_tu)u^{-1}\big) dt + \dashint_{0}^\infty  \tr_h\big((\partial_tu)\partial_s (u^{-1})\big)dt  \notag \\
	& =  \dashint_{0}^\infty  \tr_h\big((\partial_s\partial_tu)u^{-1}\big) dt - \dashint_{0}^\infty  \tr_h\big((\partial_tu)u^{-1}(\partial_s u)u^{-1}\big)dt \notag \\
	& = \dashint_{0}^\infty  \tr_h\big((\partial_s\partial_tu)u^{-1}\big) dt + \dashint_{0}^\infty  \tr_h\big(\partial_t(u^{-1})(\partial_s u)\big)dt \notag \\
	& = \dashint_{0}^\infty  \partial_t\tr_h\big((\partial_su)u^{-1}\big) dt \notag \\
	& = \tr_h\big((\partial_sF_s)(n)\big) -  \tr_h\big((\partial_su_s)(0)u_s^{-1}(0)\big) \textup{ for $n$ sufficiently large } \notag \\
	& = 0, \notag
	\end{align}
	where $\dashint$ stands for  $\frac{1}{2\pi i}\int$.  It follows that $\tau_h(w)  = \tau_h(u_0) = \tau_h(u_1)$. On the other hand, $v$ and $u_1$ differ by a local loop $\varphi$. By Lemma $\ref{lm:local}$, a local loop $\varphi\colon S^1\to  (\mathscr B(\widetilde M)^\Gamma)^+$  is homotopic to a loop $e^{2\pi i\theta}p + (1-p)$ with the propagation of the idempotent $p$ being sufficiently small.  It follows that
	\[  \frac{1}{2 \pi i} \int_0^1 \tr_h(\dot\varphi(\theta)\varphi^{-1}(\theta)) d\theta  = \int_0^1 \tr_h(p) d\theta = 0,   \]
	since the propagation of $p$ is sufficiently small. Therefore, we have  $\tau_h(v) = \tau_h(u_1) = \tau_h(w)$. This finishes the proof. 
\end{proof}

The determinant map $\tau_h \colon K_1( C^\ast_{L,0}(\widetilde M)^\Gamma)  \to \mathbb C$ 
is related to the trace map 
\[ \tr_h\colon K_0(C_r^\ast(\Gamma))\to \mathbb C \] as follows. 
\begin{lemma}\label{lm:bdry} With the same notation as above,  if  the conjugacy class $\langle h\rangle$ of a non-identity element $h\in \Gamma $ has polynomial growth, then the following diagram commutes:
	\[\xymatrix{ K_0(C_r^\ast(\Gamma)) \ar[d]_{-\tr_h} \ar[r]^-\partial  & K_1(C_{L,0}^\ast(\widetilde M)^\Gamma) \ar[d]^{\tau_h} \\
		\mathbb C \ar[r]& \mathbb C }  \]
	where $\partial \colon K_0(C_r^\ast(\Gamma))  \to  K_1(C_{L,0}^\ast(\widetilde M)^\Gamma)$ is the connecting map in the six-term $K$-theory long exact sequence for the short exact sequence:
	\[ 0 \to C_{L,0}^\ast(\widetilde M)^\Gamma \to  C_L^\ast(\widetilde M)^\Gamma  \to  C^\ast(\widetilde M)^\Gamma \to 0.\] 
\end{lemma}
\begin{proof}
	For each $[p]\in  K_0(C_r^\ast(\Gamma))$, recall that $\partial [p]$	is defined as follows: let $\{a(t)\}_{t\in [0, \infty)}$ be a lift of $p$ in $\mathscr B_L(\widetilde M)^\Gamma$ such that $a(t) = 0$ for all $t\geq 1$, in particular, $a(0) = p$,  then 
	\[  \partial p \coloneqq u \textup{ with }   u(t) = e^{2\pi i a(t)} \textup{ for } t\in [0, \infty). \]
	It follows that 
	\[ \tau_h(\partial p) =  \frac{1}{2\pi i}\int_{0}^\infty \tr_h\big(\dot{u}(t)u^{-1}(t)\big) dt = \int_0^\infty \tr_h(\dot a(t))dt = - \tr_h(p).\]
	More precisely, by our construction of the map $\tau_h$, we need to choose a regularized representative of $u$ as in Proposition $\ref{prop:reg}$. Since $u(t) = 1$ for all $t\geq 1$, a regularized representative of $u$ can tautologically\footnote{For example, the local index of the Dirac operator on the empty set gives a constant path of invertible elements $w(t) = 1$.}  be chosen to be itself. 
\end{proof}

\begin{remark}\label{rk:app}
	The same method in this section can be also applied to the following (relative) traces defined on $C_r^\ast(\Gamma)$ or $C_{\max}^\ast(\Gamma)$. Throughout this remark, we do not assume any growth conditions on the group $\Gamma$. 
\begin{enumerate}[(1)]
	\item Let $\sigma_1 \colon \Gamma \to U(n)$ and $\sigma_2 \colon \Gamma \to U(n)$ be two unitary representations of $\Gamma$ of the same dimension. They induces traces $\tr_{\sigma_i}$ on $C_{\max}^\ast(\Gamma)$ by 
	\[  \gamma \mapsto \tr(\sigma_i(\gamma)).  \]
	By using regularized representatives as in Proposition $\ref{prop:reg}$, the relative trace $\tr_{\sigma_1} -\tr_{\sigma_2}$ induces a homomorphism $\tau_{\sigma_1, \sigma_2} \colon K_1(C^\ast_{L,0}(\widetilde M)_{\max}^\Gamma) \to \mathbb C$ by 
	\[ \tau_{\sigma_1, \sigma_2}(u) = \frac{1}{2\pi i}\int_{0}^\infty (\tr_{\sigma_1}-\tr_{\sigma_2})\big(\dot{u}(t)u^{-1}(t)\big) dt. \] A key observation here is again that there exists $\varepsilon >0$ such that\footnote{To be precise, since $C_{\max}^\ast(\widetilde M)^\Gamma \cong C^\ast_{\max}(\Gamma)\otimes \mathcal K$, one needs to pass to an appropriate smooth dense subalgebra of $C^\ast_{\max}(\Gamma)\otimes \mathcal K$ on which the traces $\tr_{\sigma_1}$ and $\tr_{\sigma_2}$ are defined. Such smooth dense subalgebras always exist.  }
	\[ (\tr_{\sigma_1}-\tr_{\sigma_2})(a) = 0 \]
	if the propagation of $a \in C_{\max}^\ast(\widetilde M)^\Gamma$ is less than $\varepsilon$. Combined with the finite propagation speed of wave operators,  this provides a conceptual approach to some results of Keswani \cite{MR1763959},  Piazza and Schick \cite{MR2294190} and  Higson and Roe \cite{MR2761858}. We shall present the details in a separate paper \cite{tang-xie-yao-yu2}.
	 
	\item Let $\nu$ be the $L^2$-trace on the group von Neumann algebra  $\mathcal N\Gamma$ of $\Gamma$. It induces a trace, still denoted by $\nu$,  on $C^\ast_{\max}(\Gamma)$ by the natural map $C^\ast_{\max}(\Gamma) \to C^\ast_r(\Gamma) \to  \mathcal N\Gamma$. Now suppose  $\lambda\colon C^\ast_{\max}(\Gamma) \to \mathbb C$ is the trivial representation. Then the formula
		\[ \rho_{(2)}(u) \coloneqq  \frac{1}{2\pi i}\int_{0}^\infty (\nu-\lambda)\big(\dot{u}(t)u^{-1}(t)\big)  dt \]
		defines a  homomorphism $\rho_{(2)} \colon K_1(C^\ast_{L,0}(\widetilde M)_{\max}^\Gamma) \to \mathbb C$, which is precisely the $L^2$-$\rho$-invariant of Cheeger and Gromov \cite{MR806699}. This provides a more conceptual approach to some results of Keswani \cite{MR1794283} and Benameur and Roy \cite{MR3296587}. Again, the details will be given in \cite{tang-xie-yao-yu2}. 
\end{enumerate}
\end{remark}

\section{Higher rho invariants and  delocalized eta invariants}\label{sec:deloc}

In this section, we shall establish a precise connection between higher rho invariants and delocalized eta invariants. More precisely, let $M$ be  an odd dimensional\footnote{The even dimensional case is completely parallel. For simplicity, we will only discuss the odd dimensional case here.} closed spin manifold equipped with a positive scalar curvature metric. Denote its fundamental group $\pi_1 M$ by $\Gamma$.  Suppose $\widetilde M$ is the universal cover of $M$. We denote the Riemannnian metric lifted to $\widetilde M$ by $\tilde g$. Then the Dirac operator $\widetilde D$ on $\widetilde M$ with respect to $\tilde g$ naturally defines a higher rho invariant $\rho(\widetilde D, \tilde g) \in K_1(C_{L,0}^\ast(\widetilde M)^\Gamma)$. We shall show that if the conjugacy class $\langle h\rangle$ of a non-identity element $h\in \Gamma$ has polynomial growth, then $\tau_h(\rho(\widetilde D, \tilde g))$ is equal to the delocalized eta invariant of Lott. 

Let us briefly recall the construction of $\rho(\widetilde D, \tilde g)\in K_1(C_{L, 0}^\ast(\widetilde M)^\Gamma)$. Recall that
\[ \widetilde D^2 = \nabla^\ast \nabla + \frac{\kappa}{4}, \]
where $\nabla\colon C^\infty(\widetilde M, S) \to C^\infty(\widetilde  M, T^\ast \widetilde M\otimes S)$ is the connection on the  spinor bundle $S$ over $\widetilde  M$, $\nabla^\ast$ is the adjoint of $\nabla$, and $\kappa$ is the scalar curvature of the metric $\tilde g$. By assumption, $\kappa >\varepsilon$ for some $\varepsilon >0$, it follows immediately that $ \widetilde  D$ is invertible in this case. We define
\[  F = \widetilde D|\widetilde D|^{-1}. \]
Now for each $n\in \mathbb N$, let  $\{U_{n, j}\}$ be a $\Gamma$-invariant locally finite open cover\footnote{If $n=0$, we choose the open cover to be $\{\widetilde M\}$ consisting of a single open set $\widetilde M$ itself.} of $\widetilde M$ with $\textup{diameter}(U_{n,j}) < 1/n$ and $\{\phi_{n, j}\}$ a $\Gamma$-invariant partition of unity subordinate to $\{U_{n, j}\}$.   We define 
\begin{equation}
F(t) = \sum_{j} (1 - (t-n)) \phi_{n, j}^{1/2} F \phi_{n,j}^{1/2} + (t-n) \phi_{n+1, j}^{1/2} F \phi_{n+1, j}^{1/2}
\end{equation}
for $t\in [n, n+1]$.  Form the path of unitaries 
\[ u(t) = e^{2\pi i \frac{F(t)+1}{2}}, 0\leq t < \infty.\]
 Note that $\frac{F+1 }{2}$ is a genuine projection, hence  $u(0) = 1$. So the path $u(t), 0\leq t < \infty,$ defines a class in $K_1(C_{L, 0}^\ast(M)^\Gamma)$.

\begin{definition}
	The higher rho invariant $\rho(\widetilde D, \tilde g)$ is defined to be the $K$-theory class
	\[ [u]\in K_1(C_{L, 0}^\ast(M)^\Gamma).\]
	
\end{definition}

Now let us also recall the definition of delocalized eta invariants due to Lott.  
\begin{definition}[{\cite[Definition 7]{MR1726745}}]
	With the above notation, the delocalized eta invariant of $\widetilde D$ at $\langle h \rangle$ is defined to be 
\begin{equation}
\eta_{\langle h\rangle}(\widetilde D) \coloneqq  \frac{2}{\sqrt \pi} \int_{0}^{\infty}  \tr_h(\widetilde D e^{-t^2 \widetilde D^2})dt. 
\end{equation} 
\end{definition}
Here the convergence of the integral does not hold in general, and  relies on the growth rate of the conjugacy class $\langle h \rangle$, cf. \cite[Section 3]{MR2366359} for a more thorough discussion. A sufficient condition for the convergence of the integral is that the conjugacy class $\langle h \rangle$ has polynomial growth.

We have the following main result of this section. 
\begin{theorem}\label{thm:rhoeta}
	Let $M$ be a closed odd-dimensional spin manifold equipped with a positive scalar curvature metric $g$. Suppose $\widetilde M$ is the universal cover of $M$, $\tilde g$ is the Riemannnian metric on $\widetilde M$  lifted from $g$, and $\widetilde D$ the associated Dirac operator. Suppose the conjugacy class $\langle h\rangle$ of a non-identity element $h\in \pi_1(M)$ has polynomial growth.  Then we have
	\[ \tau_h (\rho(\widetilde D, \widetilde g) ) = - \frac{1}{2}\eta_{\langle h\rangle}(\widetilde D).  \]
\end{theorem}

 Before we prove the theorem, let us point out that  the definition of higher rho invariant $\rho(\widetilde D, \tilde g)$ does not require any growth condition on $\langle h\rangle$ or $\pi_1 M$. In fact, if the strong Novikov conjecture holds for $\Gamma = \pi_1(M)$, then we can generalize Lott's delocalized eta invariant \emph{without} any growth conditions of the conjugacy class of $h$. This can be achieved by using the Novikov rho invariant introduced in \cite[Section 7]{Weinberger:2016dq}. Let us briefly recall the construction below, and refer the reader to \cite[Section 7]{Weinberger:2016dq} for more details. Consider the following  commutative diagram: 
 \begin{equation}
 \begin{split}
 \xymatrix{  K^\Gamma_{1}(\underline E\Gamma, \widetilde M)  \ar[r] \ar[d]_{\Lambda}  & K_{0}^\Gamma(\widetilde M) \ar[r] \ar[d]^{\cong}  &  K_{0}^\Gamma(\underline E\Gamma) \ar[r] \ar[d]^{\mu_\ast} & K^\Gamma_{0}(\underline E\Gamma, \widetilde M)  \ar[d]_{\Lambda} \\
 	K_{0}(C_{L, 0}^\ast(\widetilde M)^\Gamma) \ar[r]  & K_{0}(C_{L}^\ast(\widetilde M)^\Gamma) \ar[r] & K_{0}(C^\ast_r(\Gamma))  \ar[r]^-\partial   &  K_{1}(C_{L, 0}^\ast(\widetilde M)^\Gamma)  
 } 
 \end{split}
 \end{equation}
 where $\underline{E}\Gamma$ is the universal space for proper $\Gamma$-actions and  $K^\Gamma_{i}(\underline E\Gamma, \widetilde M) $ is the  $\Gamma$-equivariant relative $K$-homology group for the pair of $\Gamma$-spaces $(\underline E\Gamma, \widetilde M)$. Let us assume  that  $\mu_\ast\colon K_{i}^\Gamma(\underline{E}\Gamma) \to K_{i}(C^\ast_r(\Gamma))$
 is a split injection\footnote{So far, in all known cases where the strong Novikov conjecture holds, the split injectivity of the Baum-Connes assembly map is known to be true as well.}.  In this case, let us denote the splitting map by $\alpha\colon K_{0}(C^\ast_r(\Gamma))\to K_{0}^\Gamma(\underline{E}\Gamma)$ which induces a direct sum decomposition:
 \[  K_{0}(C^\ast_r(\Gamma))  \cong  K_{0}^\Gamma(\underline{E}\Gamma) \oplus \mathscr E. \]  A routine diagram chase shows that
 \begin{enumerate}[(1)]
 	\item the homomorphism $\Lambda\colon K^\Gamma_{0}(\underline E\Gamma, \widetilde M) \to   K_{1}(C_{L, 0}^\ast(\widetilde M)^\Gamma) $
 	is also an injection;
 	\item and $\partial(\mathscr E)\cap \partial(K_{0}^\Gamma(\underline{E}\Gamma) ) = 0 $.   
 \end{enumerate} 
 It follows that  we have the following commutative diagram: 
 \begin{equation}\label{diag:tophr2}
 \begin{split}
 \scalebox{0.9}{\xymatrix{    K_{0}^\Gamma(\widetilde M) \ar[r] \ar[d]^{\cong}  &  K_{0}^\Gamma(\underline E\Gamma) \ar[r] \ar[d]_{\mu_\ast} & K^\Gamma_{0}(\underline E\Gamma, \widetilde M)  \ar[d]^\Lambda \ar[r] & K_{1}^\Gamma (\widetilde M)  \ar[d]^{\cong} \\
 			K_{0}(C_{L}^\ast(\widetilde M)^\Gamma) \ar[r] \ar[d]^{=} & K_{0}^\Gamma(\underline{E}\Gamma) \oplus \mathscr E  \ar[r]^-\partial  \ar[d]^\alpha &  K_{1}(C_{L, 0}^\ast(\widetilde M)^\Gamma) \ar[d]^q \ar[r]	& K_{1}(C_{L}^\ast(\widetilde M)^\Gamma) \ar[d]^{=}  \\
 		 	K_{0}(C_{L}^\ast(\widetilde M)^\Gamma) \ar[r] & K_{0}^\Gamma(\underline{E}\Gamma) \ar[r]^-\partial   &  K_{1}(C_{L, 0}^\ast(\widetilde M)^\Gamma)/\partial(\mathscr E) \ar[r] & 	K_{1}(C_{L}^\ast(\widetilde M)^\Gamma)  &	} 
 }
 \end{split}
 \end{equation}
 where $q$ is the quotient map 
 \[ q\colon  K_{1}(C_{L, 0}^\ast(\widetilde M)^\Gamma)\to  K_1(C_{L, 0}^\ast(\widetilde M)^\Gamma)/\partial(\mathscr E).  \]
 Note that the last row in diagram $\eqref{diag:tophr2}$ is also a long exact sequence. 
 By the five lemma, it follows that the composition \[ q\circ \Lambda \colon K^\Gamma_{0}(\underline E\Gamma, \widetilde M) \xrightarrow{\ \cong \ } K_{1}(C_{L, 0}^\ast(\widetilde M)^\Gamma)/\partial(\mathscr E)\] is an isomorphism.
Now we define 
 \[ \beta \coloneqq  (q\circ \Lambda)^{-1} \circ q \colon K_{1}(C_{L, 0}^\ast(\widetilde M)^\Gamma) \to   K^\Gamma_{0}(\underline E\Gamma, \widetilde M).  \]

Let  $E\Gamma$ be the universal space for free and proper $\Gamma$-actions.  By Composing $\beta$ with the natural morphism  $$ K^\Gamma_{0}(\underline E\Gamma, \widetilde M) \to K^\Gamma_{0}(\underline E\Gamma, E\Gamma)$$ induced by the inclusion $(\underline E\Gamma, \widetilde M) \hookrightarrow (\underline E\Gamma, E\Gamma)$ and the Chern character map \[ K^\Gamma_{0}(\underline E\Gamma, E\Gamma) \to \bigoplus_{k\in \mathbb Z}H^\Gamma_{2k}(\underline E\Gamma, E\Gamma)\otimes \mathbb C,\] 
we get a morphism
 \[ \Theta \colon K_{1}(C_{L, 0}^\ast(\widetilde M)^\Gamma) \to  \bigoplus_{k\in \mathbb Z}H^\Gamma_{2k}(\underline E\Gamma, E\Gamma)\otimes \mathbb C.\] 
 Recall that (cf. \cite{MR928402}, \cite[Section 7]{BCH94})
 \[  \bigoplus_{k\in \mathbb Z}H_{2k}(\underline E\Gamma, E\Gamma)\otimes \mathbb C \quad \cong \bigoplus_{\substack{\langle \gamma\rangle  \\ \gamma \textup{ finite order and } \\  \gamma \neq e } } \bigoplus_{k\in \mathbb Z}H_{2k}(Z_\gamma; \mathbb C),   \]
where $e$ is the identity element of $\Gamma$, $\langle \gamma\rangle $ runs through all conjugacy classes of finite order elements $\gamma$ with $\gamma \neq e$, and $Z_\gamma$ is the  centralizer group of $\gamma$ in $\Gamma$. 
 
\begin{definition}
If the Baum-Connes assembly map for $\pi_1(M)$ is a split injection\footnote{We remark that this definition of generalized delocalized eta invariants depends on the choice of the split injection in general.}, then we  define the generalized delocalized eta invariant $\widetilde D$ at $\langle h\rangle$ to be the complex number in the $H_0(Z_h; \mathbb C)$-component of $\rho(\widetilde D, \tilde g)$ under the map $\Theta$. In particular, if $h$ has infinite order, then the generalized delocalized eta invariant $\widetilde D$ at $\langle h\rangle$ is defined to be zero. 
\end{definition}

Note that if the conjugacy class $\langle h\rangle$ has polynomial growth and in addition the Baum-Connes assembly map is an \emph{isomorphism}, then the above generalized delocalized eta invariant coincides with the delocalized eta invariant of Lott. Indeed, in this case, the map  
\[ \Theta \colon K_{1}(C_{L, 0}^\ast(E\Gamma)^\Gamma)\otimes \mathbb C \to  \bigoplus_{k\in \mathbb Z}H^\Gamma_{2k}(\underline E\Gamma, E\Gamma)\otimes \mathbb C\] 
is an isomorphism. It follows from Lemma $\ref{lm:bdry}$ and Theorem $\ref{thm:rhoeta}$ that the above generalized delocalized eta invariant coincides with the delocalized eta invariant of Lott.

Now let us proceed to prove Theorem $\ref{thm:rhoeta}$.  In order to make the exposition more transparent, we shall work with the following alternative smooth dense subalgebra of $C^\ast(\widetilde M)^\Gamma = C^\ast_r(\Gamma)\otimes \mathcal K$.

Let $\mathscr S(\widetilde M)^\Gamma$ be the convolution algebra of all elements $A\in C^\infty(\widetilde M\times \widetilde M)$ satisfying  
\begin{enumerate}[(1)]
	\item $A$ is $\Gamma$-invariant, that is, $A(gx, gy) = A(x, y)$ for all $g\in \Gamma$, 
	\item $A$ has finite propagation, that is, there exists $R>0$ such that $A(x, y) = 0$ for all $x, y\in \widetilde M$ with $d(x, y) \geq R$. 
\end{enumerate}
The algebra $\mathscr S(\widetilde M)^\Gamma$ acts on $L^2(\widetilde M)$ by 
\[  (Af)(x) = \int_{\widetilde M} A(x, y)f(y)dy, \]
for $A \in \mathscr S(\widetilde M)^\Gamma$ and $f\in L^2(\widetilde M)$.

Fix a point $x_0\in \widetilde M$ and  let $\sigma  \colon \widetilde M\to \mathbb R$ be the distance function $\sigma(x) = d(x, x_0)$ on $\widetilde M$. In fact, we shall choose a smooth approximation $\sigma_1$ of  $\sigma$ such that $|\sigma_1(x) - \sigma(x)| < 1$ and $\|d\sigma_1(x)\| \leq 2$ for all $x\in \widetilde M$. For notational simplicity, we shall continue to denote this modified distance function by $\sigma$. Multiplication by the function $\sigma$ acts as an unbounded operator on $L^2(\widetilde M)$.  Taking commutator with $\sigma$ defines a derivation on $\mathscr S(\widetilde M)^\Gamma$:
\[ \widetilde\partial = [\sigma, \cdot] \colon  \mathscr S(\widetilde M)^\Gamma \to \mathscr S(\widetilde M)^\Gamma.   \]

Now let $\nDelta$ be the Laplace operator  on $\widetilde M$ and $r$ an integer  $> \dim M$. We define 
\begin{equation}\label{eq:alg}
\mathscr A(\widetilde M)^\Gamma = \{ A\in C^\ast(\widetilde M)^\Gamma \mid \widetilde \partial^k(A) \circ (\nDelta+1)^{r} \textup{ is bounded for } \forall k \in \mathbb N\}. 
\end{equation} 
The same proof from \cite[Lemma 6.4]{CM90} shows that $\mathscr A(\widetilde M)^\Gamma$ contains $\mathscr S(\widetilde M)^\Gamma$  and is closed under holomorphic functional calculus. 

We associate to each element $h\in \Gamma$ the following trace on $\mathscr S(\widetilde M)^\Gamma$  : 
\[ \tr_h(A) = \sum_{g\in \langle h \rangle} \int_{\mathcal F} A(x, gx) dx   \] 
where $\mathcal F$ is a fundamental domain of $\widetilde M$ under the action of $\Gamma$.
Here we have identified $L^2(\widetilde M)$ with $L^2(\mathcal F) \otimes \ell^2(\Gamma)$ through the mapping $f \to \hat f$ by the formula $\hat f(x, \alpha) = f(\alpha x)$ for $x \in \mathcal F$ and $\alpha \in \Gamma$. In particular, each element $A\in \mathscr S(\widetilde M)^\Gamma$ becomes a finite sum $\sum_{g\in \Gamma} (A_g) R_g$, where $
	A_g(x, y) = A(x, gy)$ for $x, y\in \mathcal F$ and  $R$ denotes the right regular representation of $\Gamma$.


The following  lemma and its proof are essentially the same as Lemma $\ref{lm:tr}$. We shall be brief.  
\begin{lemma}
	If $\langle h\rangle $ has polynomial growth, then	$\tr_h$ extends to a continuous trace on  $\mathscr A(\widetilde M)^\Gamma$.  
\end{lemma}
\begin{proof}
	Let $A$ be an element in $\mathscr A(\widetilde M)^\Gamma$.  By assumption, $\widetilde \partial^k(A)\circ (\nDelta+1)^{r}$ is bounded for all $k\in \mathbb N$.   It follows from Sobolev embedding theorem that the Schwartz kernel $\widetilde \partial^k(A)(x, y)$  of $\widetilde \partial^k(A)$ is a uniformly bounded continuous function on $\widetilde M \times \widetilde M$ for each $k\in \mathbb N$.   We define 
	\begin{equation}
	\tr_h(A) =  \sum_{g\in \langle h \rangle} \int_{\mathcal F} A(x, gx) dx 
	\end{equation} 
	We need to verify that the summation on the right side converges. Observe that
	\[  \widetilde \partial^{2k} (A)(x, gx) = \big(\sigma(x) - \sigma(gx)\big)^{2k}A(x, gx), \]
	and furthermore $|\sigma(x) - \sigma(gx)|\geq |g| -\textup{diam}(\mathcal F)$ for all $x\in \mathcal F$, where $\textup{diam}(\mathcal F)$ is the diameter of $\mathcal F$. 	It follows that there exists a fixed constant $C_k >0$ such that 
	\begin{equation}
	(1+|g|)^{4k}\int_{\mathcal F} |A(x, gx)|^2 dx  \leq C_k\int_{\mathcal F} |\widetilde \partial^{2k} (A)(x, gx)|^2 dx. 
	\end{equation}
	Now 
	for all $k\in \mathbb N$, we have 
	\begin{align*}
	|\tr_h(A)|  & \leq \sum_{g\in \langle h \rangle} \int_{\mathcal F} | A(x, gx)| dx    \\
	&  \leq \Big(\sum_{g\in \langle h \rangle} (1+|g|)^{2k}\int_{\mathcal F}|A(x, gx)|^2dx\Big)^{1/2}\Big( \sum_{g\in \langle h \rangle} (1+|g|)^{-2k}\Big)^{1/2}\\
	& \leq C_k^{1/2} \Big(\sum_{g\in \langle h \rangle} (1+|g|)^{-2k}\int_{\mathcal F}|\widetilde \partial^{2k}(A)(x, gx)|^2dx\Big)^{1/2}\Big( \sum_{g\in \langle h \rangle} (1+|g|)^{-2k}\Big)^{1/2},
	\end{align*}		
	which is finite for $k$ sufficiently large, since $\langle h\rangle $ has polynomial growth and $\widetilde \partial^{2k}A(x, gx)$ is uniformly bounded for all $g$. 
	Now a similar argument\footnote{Note that there exists a fixed constant $\Lambda_{j}$ such that  the supremum norm of the continuous function $\widetilde \partial^{j}(A)(x, gy)$ is $\leq  \Lambda_{j}\cdot  \|\widetilde \partial^{j}(A)\circ (\nDelta+1)^{r}\|$ for all $A\in \mathscr A(\widetilde M)^\Gamma$. } as in Lemma $\ref{lm:tr}$ shows that $\tr_h$ is a continuous map and  $\tr_h(AB) = \tr_h(BA)$ for $A, B\in \mathscr A(\widetilde M)^\Gamma$.  
\end{proof}

Let $\mathcal S$ be the associated spinor bundle on $\widetilde M$ and $\mathcal S^\ast$ its dual bundle. Consider the bundle $\textup{End}(\mathcal S) = p_1^\ast (\mathcal S)\otimes p_2^\ast (\mathcal S^\ast)$ on $\widetilde M\times \widetilde M$, where $p_i\colon \widetilde M\times \widetilde M \to \widetilde M$ is the projection onto the first and second component respectively. There is a natural diagonal action of $\Gamma$ on $\textup{End}(\mathcal S)$.   Define $C^\infty(\widetilde M\times \widetilde M, \textup{End}(\mathcal S))$ to be the set of all smooth sections of the bundle $\textup{End}(\mathcal S)$ over $\widetilde M\times \widetilde M$.   Let $\mathscr S(\widetilde M, \mathcal S)^\Gamma$ be the convolution algebra of all $\Gamma$-invariant finite propagation elements in $C^\infty(\widetilde M\times \widetilde M, \textup{End}(\mathcal S))$. The algebra $\mathscr S(\widetilde M, \mathcal S)^\Gamma$ acts on $L^2(\widetilde M, \mathcal S)$ by 
\[  (Af)(x) = \int_{\widetilde M} A(x, y)f(y)dy, \]
for $A \in \mathscr S(\widetilde M, \mathcal S)^\Gamma$ and $f\in L^2(\widetilde M, \mathcal S)$, where $L^2(\widetilde M, \mathcal S)$ is the space of $L^2$-sections of $\mathcal S$ over $\widetilde M$.

Recall that the $\Gamma$-equivariant Roe algebra $C^\ast(\widetilde M)^\Gamma$ is (up to isomorphism) independent of the choice of admissible $(\widetilde M, \Gamma)$-modules. We shall still denote the $\Gamma$-equivariant Roe algebra obtained from the $(\widetilde M, \Gamma)$-module $L^2(\widetilde M, \mathcal S)$ by $C^\ast(\widetilde M)^\Gamma$.  Similar to line $\eqref{eq:alg}$,  we define 
\[ \mathscr A(\widetilde M, \mathcal S)^\Gamma = \{ A\in C^\ast(\widetilde M)^\Gamma \mid \widetilde \partial^k(A) \circ (\widetilde D^{2n} + 1) \textup{ is bounded for } \forall k \in \mathbb N\}. \] 
where $\widetilde D$ is the Dirac operator on $\widetilde M$ and $n$ is a fixed integer $> \dim M$.
As before, we can similarly define the algebras $\mathscr A_L(\widetilde M, \mathcal S)^\Gamma$ and $\mathscr A_{L,0}(\widetilde M, \mathcal S)^\Gamma$ as follows 

\begin{definition}\label{def:alg}
	We define $\mathscr A_L(\widetilde M, \mathcal S )^\Gamma$ to be the dense subalgebra of $C_L^\ast(\widetilde M)^\Gamma$ consisting of elements $f\in C_L^\ast(\widetilde M)^\Gamma$ such that  $f$ is piecewise smooth and $f(t) \in \mathscr A(\widetilde M, \mathcal S)^\Gamma$ for all $t\in [0, \infty)$. Also, we define $\mathscr A_{L, 0}(\widetilde M, \mathcal S)^\Gamma$ to be the kernel of the evaluation map 
	\[ \ev\colon  \mathscr A_L(\widetilde M, \mathcal S )^\Gamma \to  \mathscr A(\widetilde M, \mathcal S)^\Gamma \textup{ defined by } f\mapsto f(0).\]
\end{definition}

Recall that a smooth function $\varphi$ on $\mathbb R$ is called a Schwartz function if the function $x^k \varphi^{(j)}(x)$  is bounded on $\mathbb R$ for each $k, j\in \mathbb N$, where $\varphi^{(j)}$ is the $j$-th derivative of $\varphi$.    

\begin{proposition}\label{prop:schwartz}
	Suppose $\varphi$ is a Schwartz function, then $\varphi(\widetilde D)$ is an element in $ \mathscr A(\widetilde M, \mathcal S)^\Gamma$.  
\end{proposition}
\begin{proof}
	Recall that 
	\[ \varphi(\widetilde D) = \frac{1}{2\pi}\int_{-\infty}^\infty \hat{\varphi}(s)e^{is\widetilde D} ds, \]
	where $\hat\varphi$ is the Fourier transform of $\varphi$, which is also a Schwartz function, since $\varphi$ is. 
	Let us first verify that  $[\varphi(\widetilde D), \sigma]$ is a bounded operator. 
	Consider 
	\[ f(s) = e^{is\widetilde D} A e^{-is\widetilde D} - A. \]
	Then 
	\[  f^\prime(s) =  e^{is\widetilde D} i\widetilde DA e^{-is\widetilde D} - e^{is\widetilde D} A(i\widetilde D) e^{-isD}= ie^{is\widetilde D} [\widetilde D, A] e^{-is\widetilde D}\]
	and $f(0) = 0$, which implies that
	\[ f(s) = e^{is\widetilde D} A e^{-is\widetilde D} - A =   i\int_0^s e^{it\widetilde D} [\widetilde D, A] e^{-it\widetilde D}dt,\]
	or equivalently,  
	\[ [e^{is\widetilde D}, A] =   i\int_0^s e^{it\widetilde D} [\widetilde D, A] e^{i(s-t)\widetilde D}dt.  \]
	It follows that 
	\[ [\varphi(\widetilde D), \sigma] = \frac{i}{2\pi}\int_{-\infty}^\infty \hat{\varphi}(s)\int_0^s e^{it\widetilde D} [\widetilde D, \sigma] e^{i(s-t)\widetilde D}dt ds. \]
	Note that $\|e^{it\widetilde D} [\widetilde D, \sigma] e^{i(s-t)\widetilde D}\| \leq \big\|[\widetilde D, \sigma]\big\|$ for all $s, t\in \mathbb R$. We see that  
	\[ \left\|\int_0^s e^{it\widetilde D} [\widetilde D, \sigma] e^{i(s-t)\widetilde D}dt\right\| \leq s\cdot \big\|[\widetilde D, \sigma]\big\|. \]
	This implies that $[\varphi(\widetilde D), \sigma]$ has finite operator norm, since the Fourier transform $\hat{\varphi}$ is a Schwartz function. Now a straightforward inductive argument shows that $\widetilde \partial^k (\varphi(\widetilde D))$ is a  bounded operator for each $k\in \mathbb N$. 
	
	Now let us show that $\widetilde \partial^k(\varphi(\widetilde D))\circ \widetilde D^{2n}$ is bounded. Note that
	\[  [\sigma, \varphi(\widetilde D)]\widetilde D^{2n} = [\sigma,\varphi(\widetilde D) \widetilde D^{2n}] -  \varphi(\widetilde D) [\sigma, \widetilde D^{2n}].  \]
	By the same argument as above, the operator $[\sigma, \varphi(\widetilde D) \widetilde D^{2n}]$ is bounded, since  $\varphi(\widetilde D) \widetilde D^{2n} = \psi(\widetilde D)$, where $\psi(x) = \varphi(x) x^{2n}$  is a Schwartz function. Also, observe that 
	\begin{align*}
	\varphi(\widetilde D) [\sigma, \widetilde D^{2n}] &= \sum_{k=0}^{n-1}\varphi(\widetilde D)\widetilde D^k [\sigma, \widetilde D]\widetilde D^{n-k-1} \\
	& = \sum_{k=0}^{n-1}\left(\varphi(\widetilde D)\widetilde D^k (\widetilde D^{2n} +1)\right)  (\widetilde D^{2n} +1)^{-1}[\sigma, \widetilde D]\widetilde D^{2n-k-1}.
	\end{align*}  
	Note that $(\widetilde D^{2n} +1)^{-1}[\sigma, \widetilde D]\widetilde D^{2n-k-1}$ is bounded for all $ 0\leq k \leq 2n-1$. It follows that $\varphi(\widetilde D) [\sigma, \widetilde D^{2n}]$ is bounded. This finishes the proof. 
\end{proof}

Now let us prove Theorem $\ref{thm:rhoeta}$. 

\begin{proof}[Proof of Theorem $\ref{thm:rhoeta}$]
	Recall that  a smooth normalizing function $\phi\colon \mathbb R \to [-1, 1]$ is an odd function such that  $\phi(x) >0$ when $x>0$,  and $\phi (x) \to \pm$ as $x \to \pm \infty$. Consider the normalizing function 
	\[ \varphi(x) = \frac{2}{\sqrt \pi}\int_{0}^x e^{-s^2}ds.   \]
	Define $F(t) = \varphi(t^{-1}\widetilde D)$ for $t\in (0, \infty)$. Since the scalar curvature of $\widetilde g$ is uniformly bounded below by a positive number, it follows that $\widetilde D$ is invertible. In particular, there is a spectral gap near $0$ in the spectrum of $\widetilde D$. This implies that $F(t)$ converges to $\textup{sign}(\widetilde D) = \widetilde D|\widetilde D|^{-1}$ in operator norm, as $t\to 0$. Define $F(0) = \textup{sign}(\widetilde D)$. The path 
	\[  u(t) = e^{2\pi i \frac{F(t) + 1}{2}} \textup{ with } t\in [0, \infty) \]
	defines an element in $K_1(C_{L, 0}^\ast(\widetilde M)^\Gamma)$. Indeed, by construction, we have $u(0) = 1$. Moreover, we have 
		\[ \varphi(\widetilde D) = \frac{1}{2\pi}\int_{-\infty}^\infty \hat{\varphi}(s)e^{is\widetilde D} ds, \]
	where $\hat\varphi$ is the Fourier transform of $\varphi$.
Since every smooth normalizing function can be approximated (in supremum norm) by smooth normalizing functions whose distributional Fourier transform has compact support. It follows from functional calculus and finite propagation of the wave operator $e^{is\widetilde D}$  that the path $u$ can be uniformly approximated by paths of invertible elements  with finite propagation. Observe that, for each smooth normalizing function $\psi$, the operator   $\psi(\widetilde D)^2 - 1$ is locally compact, hence $\frac{\psi(\widetilde D) +1}{2}$ is a projection modulo locally compact operators. This implies that   $e^{2\pi i \frac{\psi(\widetilde D) + 1}{2}} \equiv 1$ modulo locally compact operators.   To summarize, we see that 
the path $u$ defines an invertible element in  $(C_{L, 0}^\ast(\widetilde M)^\Gamma)^+$ . Moreover, it is not difficult to see that $1 - \exp(2\pi i\frac{\varphi + 1}{2})$ is a Schwartz function.  By Proposition $\ref{prop:schwartz}$,  it follows that $u(t)  \in (\mathscr A(\widetilde M, \mathcal S)^\Gamma)^+$ for each $t\in [0, \infty)$. To show that  $u$  is an invertible element in $\mathscr (\mathscr A_{L, 0}(\widetilde M, \mathcal S)^\Gamma)^+$, it suffices to show that $u-1$ is  smooth with respect to the Fr\'echet topology of $\mathscr A(\widetilde M, \mathcal S)^\Gamma$. 

For each $N\in \mathbb N$, we define the algebra 
\[ \mathscr A_N = \{ A\in C^\ast(\widetilde M)^\Gamma \mid \widetilde \partial^k(A) \circ (\widetilde D^{2n} + 1) \textup{ is bounded for } \forall k \leq N\} \] and a norm on $\mathscr A_N $ by
\[ \|A\|_{\mathscr A_N }= \sum_{k= 0}^N \frac{1}{k!} \big\|\widetilde \partial^k(A)\circ (\widetilde D^{2n} +1)\big\|_{op}  \]
where $\|\widetilde \partial^k(A)\circ (\widetilde D^{2n} +1)\big\|_{op}$ stands for the operator norm of $\widetilde \partial^k(A)\circ (\widetilde D^{2n} +1)$.  By definition, we have $\mathscr A(\widetilde M, \mathcal S)^\Gamma = \bigcap_{N\in \mathbb N} \mathscr A_N$, which becomes a Fr\'echet algebra under this sequence of norms $\{\|\cdot \|_{\mathscr A_N }\colon  N\in \mathbb N\}$. 

Since $\widetilde \partial$ is a derivation,  it is straightforward to see that $\mathscr A_N$ is closed under holomorphic functional calculus. Therefore, it suffices to show that $u-1$ is smooth with respect to the above norm $\|\cdot \|_{\mathscr A_N}$  for each $N\in \mathbb N$. 

Recall that $\widetilde D$ is invertible. Let  $\sigma >0$ be the spectral gap of $\widetilde D$ at zero, that is, the spectrum of $\widetilde D$ is disjoint from the interval $(-\sigma, \sigma)$. Since $\mathscr A_N$ is closed under holomorphic functional calculus, the spectral radius of $e^{-\widetilde D^2}$ in $\mathscr A_N$ is $e^{-\sigma^2}$. 
Now apply the spectral radius formula
$$\lim_{n\to\infty}\big\|(e^{-\widetilde D})^n\big\|_{\mathscr A_N}^{\frac 1 n}=e^{-\sigma^2}.$$
 It follows that there exists  $C_1>0$ such that
$$\big\|e^{-s\widetilde D^2}\big\|_{\mathscr A_N}\leqslant C_1 \cdot e^{-s\sigma^2/2}$$
for all sufficiently large  $s\gg0$. 
Therefore, there exists $C >0$ such that 
\begin{equation}\label{eq:estimate}
\begin{split}
\| \pi i \dot F(t)\|_{\mathscr A_N} =&
\big\|-2 i \sqrt\pi t^{-2}\widetilde D e^{-\widetilde D^2/t^2}\big\|_{\mathscr A_N}\\
\leqslant& t^{-2}\big\|2\sqrt\pi  \widetilde D e^{-\widetilde D^2}\big\|_{\mathscr A_N}\cdot\big\| e^{-(1/t^2-1)\widetilde D^2}\big\|_{\mathscr A_N}\\
\leqslant & C t^{-2} e^{-\frac{1}{2}\sigma^2/t^2},
\end{split}
\end{equation}
where $\dot F(t)$ is the derivative of $F(t)$ with respect to $t$. 
Note that  we have 
$$u(t)-1=\exp\left(\pi i \int_0^t \dot F(r) dr \right)-1.$$ It follows that, for sufficiently small $t>0$,
\begin{equation}\label{6.4}
\begin{split}
\|u(t)-1\|_{\mathscr A_N}=&\left\|\sum_{n=1}^\infty\frac{1}{n!}\left(\int_0^t \pi i \dot F(r) dr \right)^n\right\|_{\mathscr A_N}	\\
\leqslant& \sum_{n=1}^\infty \frac{1}{n!}\left(Ct^{-2}e^{-\frac{1}{2}\sigma^2/t^2}\right )^n=	
\exp\left( C t^{-2}e^{-\frac{1}{2}\sigma^2/t^2}\right) -1,
\end{split}
\end{equation}
where the last term goes to $0$ as $t\to 0$. This implies that $u-1$ is continuous with respect to the norm $\|\cdot\|_{\mathscr A_N}$. Now take the derivatives of $u-1$, and a similar argument also shows that the derivatives of $u-1$ are continuous with respect to the norm $\|\cdot\|_{\mathscr A_N}$.  Therefore,  $u-1$ is smooth with respect to the norm $\|\cdot\|_{\mathscr A_N}$. This proves that $u$  is an invertible element in $\mathscr (\mathscr A_{L, 0}(\widetilde M, \mathcal S)^\Gamma)^+$.

	 Observe that 
	\[  \frac{1}{2\pi i}\int_{0}^\infty \tr_h( \dot u(t) u^{-1}(t)) dt  = \frac{-1}{\sqrt \pi} \int_{0}^{\infty}  \tr_h(\widetilde D e^{-t^2 \widetilde D^2})dt. \]
	It follows from the work of Lott \cite[Section 4, Page 20]{MR1726745} that the above integral  converges absolutely.

	On the other hand, our definition of $\tau_h (\rho(\widetilde D, \widetilde g) )$ is expressed in terms of  regularized representatives of $\rho(\widetilde D, \widetilde g)$ (cf. Proposition $\ref{prop:reg}$ and Definition $\ref{def:reg}$).  To be precise,  let us choose a specific regularized representative as follows.  Let $\chi$ be a smooth normalizing function such that the support of its distributional Fourier transform $\hat \chi$ has compact support. If  we denote $ G(t) = \chi(t^{-1}\widetilde D)$, then the path $w$ defined by
	\[ w(t)=\begin{cases}
	u(t) &0\leqslant t\leqslant 1\\
	e^{2\pi i\frac{(2-t)F(1)+(t-1) G(1) +1}{2} } 
	& 1\leqslant t\leqslant 2\\
	e^{2\pi i\frac{G(t-1)+1}{2}} & t\geqslant 2
	\end{cases} 
	\] 
	is a regularized representative of $[u]$ in the sense of Proposition $\ref{prop:reg}$.   Now a similar calculation as the transgression formula $\eqref{eq:trans}$ in Proposition $\ref{prop:det}$ shows that 
	\[  \frac{1}{2\pi i}\int_{0}^\infty \tr_h( \dot w_t w^{-1}_t) dt = \frac{1}{2\pi i}\int_{0}^\infty \tr_h( \dot u_t u^{-1}_t) dt. \]
	Here let us briefly comment on the convergence of various terms that appear in the  transgression formula. Observe that there exists $\varepsilon >0$ such that $\tr_h(a) = 0$ for $a\in \mathscr A(\widetilde M, \mathcal S)^\Gamma$, as long as the propagation of $a$ is less than $\varepsilon$. Now the convergence of various terms that appear in the transgression formula  follows from this observation and the work of Lott \cite[Section 4, Page 20]{MR1726745}. To summarize, we have proved that 
	\[ \tau_h (\rho(\widetilde D, \widetilde g) ) = - \frac{1}{2}\eta_{\langle h\rangle}(\widetilde D). \]
	This finishes the proof. 
\end{proof}

\section{Baum-Connes conjecture and algebraicity of delocalized eta invariants}\label{sec:bcalg}

In this section, we prove an algebraicity result concerning the values of delocalized eta invariants. We also give a $K$-theoretic proof of a version of delocalized  Atiyah-Patodi-Singer index theorem.

\begin{definition}
	Given a discrete group $\Gamma$, let $\mathbb Q_\Gamma$ be the field extension of $\mathbb Q$ by the following set of roots of unity:
	\[  \{e^{\pi i/n} \mid \textup{ there exists $\alpha \in \Gamma$ such that the order of $\alpha$ is $n$ }\}. \]
\end{definition}

We have the following algebraicity result concerning the values of delocalized eta invariants. 

\begin{theorem}
Assume the same notation as in Theorem $\ref{thm:rhoeta}$. If the rational Baum-Connes conjecture holds for $\Gamma$, that is, the assembly map $\mu\colon K_i^\Gamma(\underline{E}\Gamma) \to K_i(C_r^\ast(\Gamma))$ is a rational isomorphism, and the conjugacy class $\langle h\rangle$ of a non-identity element $h\in \Gamma$ has polynomial growth,  then
	$\eta_{\langle h\rangle}(\widetilde D)$  is an element in $\mathbb Q_\Gamma$.  If in addition $h$ has infinite order, then 	$\eta_{\langle h\rangle}(\widetilde D) = 0$.
	
\end{theorem}
\begin{proof}
	Consider the following long exact sequence: 
	\[\xymatrix{ K_0(C_{L, 0}^\ast(E\Gamma)^\Gamma)\otimes \mathbb Q \ar[r] &  K_0(C_L^\ast(E\Gamma)^\Gamma) \otimes \mathbb Q\ar[r]^-\mu  & K_0(C_{r}^\ast(\Gamma))\otimes \mathbb Q \ar[d]^\partial  \\
		K_1(C_r^\ast(\Gamma))\otimes \mathbb Q\ar[u]	& K_1(C_{L}^\ast(E\Gamma)^\Gamma)\otimes \mathbb Q\ar[l] & K_1(C_{L, 0}^\ast(E\Gamma)^\Gamma)\otimes \mathbb Q \ar[l] }  \]
	Recall that the morphism $  K_i(C_L^\ast(E\Gamma)^\Gamma)  \to K_i(C_L^\ast(\underline{E}\Gamma)^\Gamma) $ induced by the natural inclusion from $E\Gamma$ to $\underline{E}\Gamma$ is rationally injective (cf. \cite[Section 7]{BCH94}).
	It follows that, if the rational Baum-Connes conjecture holds for $\Gamma$, that is, the assembly map $\mu\colon K_i(C_L^\ast(\underline{E}\Gamma)^\Gamma) \to K_i(C_r^\ast(\Gamma))$ is rationally isomorphic,  then the map $\mu$ is injective and the map $\partial$ is surjective. In particular, rationally every element $u\in K_1( C_{L,0}^\ast(E\Gamma)^\Gamma)$ is the image of some $p \in K_0( C_r^\ast(\Gamma)) $  under the map $\partial$. By Lemma $\ref{lm:bdry}$, we have 
	\[  \tau_h(u) = -\tr_h(p). \]
	Moreover, the map $\tau_h\colon K_1( C_{L,0}^\ast(\widetilde M)^\Gamma) \to \mathbb C$ factors through   $K_1( C_{L,0}^\ast(E\Gamma)^\Gamma)$.  Therefore, 	the image of the map $\tau_h\colon K_1( C_{L,0}^\ast(\widetilde M)^\Gamma) \to \mathbb C$	is (up to multiplication by rational numbers) equal to the image of the map $\tr_h \colon K_0(C_r^\ast(\Gamma)) \to \mathbb C$. 
	
	By Theorem $\ref{thm:rhoeta}$, it suffices to show that image of the map $\tr_h \colon K_0(C_r^\ast(\Gamma)) \to \mathbb C$ is contained in $\mathbb Q_\Gamma$. By using the Baum-Douglas model of $K$-homology,  $K_0(C_L^\ast(\underline{E}\Gamma)^\Gamma)$ is generated by $(M, E, \varphi)$, where $M$ is a complete $spin^c$ manifold equipped with a proper and cocompact $\Gamma$-action, $E$ is a $\Gamma$-equivariant bundle over $M$, and $\varphi\colon M\to \underline{E}\Gamma$ is a $\Gamma$-equivariant map. In this case, the Baum-Connes assembly map takes $(M, E, \varphi)$ to its higher index $\ind_\Gamma(D_E)$, where $D_E$ is the associated Dirac operator on $M$ twisted by $E$. Now by Corollary $\ref{cor:alg}$ in the appendix of our paper,  $\tr_h(\ind_\Gamma(D_E))$ is an algebraic number in $\mathbb Q_\Gamma$, and  if in addition $h$ has infinite order, then $\tr_h(\ind_\Gamma(D_E)) =0$. This finishes the proof.

\end{proof}

In light of the above theorem, we propose the following question. 

\begin{question}
	What values can delocalized eta invariants take in general? Are they always algebraic numbers?
\end{question}

In particular, if a delocalized eta invariant is transcendental, then it would lead to a counterexample to the Baum-Connes conjecture. Note that the above question is a reminiscent of Atiyah's question concerning rationality of $\ell^2$-Betti numbers \cite{MR0420729}. Atiyah's question was answered in negative by Austin, who showed that $\ell^2$-Betti numbers can be transcendental \cite{MR3149852}.

Now let us turn to a delocalized Atiyah-Patodi-Singer index theorem. Let  $W$ be a compact $n$-dimensional spin manifold with boundary $\partial W$. Suppose $W$ is equipped with a Riemannian metric $g_W$ which has product structure near $\partial W$ and in addition has positive scalar curvature on $\partial W$. Let $\widetilde W$ be the universal covering of $W$ and $g_{\widetilde W}$  the   Riemannian metric  on $\widetilde W$ lifted from $g_W$. Denote $\pi_1(W)$ by $\Gamma$. With respect to the metric $g_{\widetilde W}$,  the associated Dirac operator $\widetilde D$ on $\widetilde  W$ naturally defines a higher index, denoted by $\ind_\Gamma(\widetilde D, g_{\widetilde W})$, in $K_n(C^\ast(\widetilde W)^\Gamma) = K_n(C_r^\ast(\Gamma))$, cf. \cite[Section 3]{Xie2014823}. Let $\tilde g_\partial$ be the restriction of $g_{\widetilde W}$ on $\partial \widetilde{W}$. As we have seen above,  with respect to the metric $\tilde g_\partial$,  the associated Dirac operator $\widetilde D_\partial$ on $\partial {\widetilde W}$ naturally defines a higher rho invariant $\rho(\widetilde D_\partial, \tilde g_{\partial} )$ in $K_{n-1}(C_{L, 0}^\ast(\partial\widetilde W)^\Gamma)$. If no confusion is likely to arise,  the image of $\rho(\widetilde D_\partial, \tilde g_{\partial} )$  in $K_{n-1}(C_{L, 0}^\ast(\widetilde W)^\Gamma)$ under the natural morphism $ K_{n-1}(C_{L, 0}^\ast(\partial\widetilde W)^\Gamma) \to K_{n-1}(C_{L, 0}^\ast(\widetilde W)^\Gamma)$ will still be denoted by $\rho(\widetilde D_\partial, \tilde g_{\partial} )$. 

We denote by $\partial\colon K_n(C^\ast(\widetilde W)^\Gamma)  \to  K_{n-1}(C_{L,0}^\ast(\widetilde W)^\Gamma)$  the connecting map in the K-theory long exact sequence induced by the short exact sequence of $C^\ast$-algebras: 
\[ 0 \to C_{L,0}^\ast(\widetilde W)^\Gamma \to  C_L^\ast(\widetilde W)^\Gamma  \to  C^\ast(\widetilde W)^\Gamma \to 0.\]  
Then we have 	\[  \partial ( \ind_\Gamma(\widetilde D, g_{\widetilde W})) =  \rho(\widetilde D_\partial, \tilde g_{\partial} ) \textup{ in }  K_{n-1}(C_{L,0}^\ast(\widetilde W)^\Gamma), \]
cf.  \cite[Theorem 1.14]{MR3286895}, \cite[Theorem A]{Xie2014823}.

We have the following version of the delocalized Atiyah-Patodi-Singer index theorem. 

\begin{proposition}\label{thm:aps}
Let  $W$ be a compact even-dimensional spin manifold with boundary $\partial W$. Suppose $W$ is equipped with a Riemannian metric $g_W$ which has product structure near $\partial W$ and in addition has positive scalar curvature on $\partial W$. If the conjugacy class $\langle h\rangle$ of a non-identity element $h\in \Gamma = \pi_1(W)$ has polynomial growth, then 
	\[ \tr_h(\ind_\Gamma(\widetilde D, g_{\widetilde W})) =  \frac{\eta_{\langle h\rangle}(\widetilde D_\partial)}{2}.  \] 
\end{proposition}

\begin{proof}
	Since 
	\[  \partial ( \ind_\Gamma(\widetilde D, g_{\widetilde W})) =  \rho(\widetilde D_\partial, \tilde g_{\partial} ) \textup{ in }  K_{1}(C_{L,0}^\ast(\widetilde W)^\Gamma), \]
	the statement follows immediately from Theorem $\ref{thm:rhoeta}$ and Lemma $\ref{lm:bdry}$. 
\end{proof}

\begin{remark}
	When $\Gamma$ is virtually nilpotent, a similar result has been proved at the level of noncommutative de Rham homology by Leichtnam and Piazza \cite[Theorem 14.1]{MR1488084}. 
\end{remark}

\appendix
\section{ $L^2$-Lefschetz fixed point theorem}

In this appendix, we shall briefly review a modified version of the $L^2$-Lefschetz fixed point theorem of B.-L. Wang and H. Wang 
 for proper actions of discrete groups on complete manifolds \cite[Theorem 5.10]{MR3454548}.

Let $X$ be a complete $spin^c$ manifold of dimension $n$, and $E$ a Hermitian vector bundle over $X$. Suppose a discrete group $\Gamma$ acts properly and cocompactly on $X$  through isometries. Moreover, assume this action lifts to actions on the associated $spin^c$ bundle $S$ and the bundle $E$ over $X$ through isometric bundle morphisms. 

For each $h\in \Gamma$, let $X^h = \{ x\in X\mid h\cdot x = x\}$ the fixed point set of $h$. Denote the normal bundle of $X^h$ by $N$. The bundle $N$ admits a decomposition 
\[  N = N(\pi) \oplus \bigoplus_{0 < \theta < \pi} N(\theta) \]
where the differential $d_h$ of the map $h$ acts on $N(\pi)$ by multiplication by $-1$, and for each $0< \theta < \pi$, $N(\theta)$ is a complex bundle in which $d_h$ acts by multiplication by $e^{i\theta}$. Since $h$ is orientation preserving, the bundle $N(\pi)$ is an  oriented even-dimensional real bundle. 

The $L^2$-Lefschetz fixed point theorem will be expressed in terms of characteristic classes as follows. We refer to \cite[Chapter III, Section 14]{BLMM89} for more details. If $V$ is a complex vector bundle with formal splitting $V = \ell_1\oplus \dots \oplus \ell_k$  into line bundles with the corresponding first Chern class denoted by $c_1(\ell_j) = x_j$, then for each $0< \theta < \pi$, we define 
\[ \hat A_{\theta}(V) = 2^{-k}\prod_{j=1}^{k}  \frac{1}{\sinh \frac{1}{2}(x_j + i\theta)} = \prod_{j=1}^{k}\frac{e^{\frac{1}{2}(x_j + i\theta)}}{e^{(x_j+i\theta)} - 1} \]
When $\theta = \pi$,  we define a characteristic class $\hat A_{\pi}(V)$ for any oriented real $2k$-dimensional bundle as follows. Let $V = V_1\oplus \dots \oplus V_k$  be a formal splitting into oriented $2$-plane bundles, and set $x_j = \chi(V_j)$ the Euler class of $V_j$. We define 
\[ \hat A_{\pi}(V) = 2^{-k}\prod_{j=1}^{k}  \frac{1}{\sinh \frac{1}{2}(x_j + i\pi)} = (2i)^{-k} \prod_{j=1}^{k}\frac{1}{\cosh(x_j/2)}. \]

Let $\ell$ be the associated line bundle for the $spin^c$-structure of $X$, and $c_1 = c_1(\ell)$ its first Chern class. Suppose $d_h$ acts on $\ell$ by multiplication by $e^{i\beta}$.

Furthermore, $X^h$ consists of  Let $D_E$ be the associated Dirac operator twisted by $E$ on $X$  and denote its higher index by $\ind_\Gamma(D_E) \in K_n(C^\ast_r(\Gamma))$.  We have the following modified version of $L^2$-Lefschetz fixed point theorem of B.-L. Wang and H. Wang \cite[Theorem 5.10 \& Theorem 6.1]{MR3454548}.

\begin{theorem}
	With the same notation as above,  if the conjugacy class $\langle h \rangle $ of $h\in \Gamma$ has polynomial growth, then 
	\begin{equation}\label{eq:loc}
	\tr_h(\ind_\Gamma(D_E)) =  \int_{\mathcal F} \Big( \prod_{0 < \theta \leq \pi} \hat A_\theta(N(\theta))\cdot \hat A(X^h) \cdot  e^{\frac{c_1 + i\beta }{2}}\cdot \textup{ch}(E)\Big). 
	\end{equation} 
Here  $ \tr_h(\ind_\Gamma(D_E))$ stands for  the evaluation of the linear map $\tr_h\colon K_0(C_r^\ast(\Gamma)) \to \mathbb C$ on the higher index class $\ind_\Gamma(D_E)$, and  $\mathcal F$ is a fundamental domain of $X^h$ under the action of  $Z_h$, where $Z_h$ is the centralizer group of $h$ in $\Gamma$. 
\end{theorem}

Although B.-L. Wang and H. Wang made the  assumption that $\tr_h$ extends to a trace map on $C_r^\ast(\Gamma)$ in \cite[Theorem 5.10]{MR3454548}, we point out that it suffices to have the less restrictive assumption that  the conjugacy class $\langle h \rangle $ of $h\in \Gamma$ has polynomial growth. Indeed, the higher index class $\ind_\Gamma(D_E)$ can be represented by elements in terms of the heat kernel operator $e^{-(\widetilde D_E)^2}$ (see for example \cite[Page 356]{CM90}). Such a representative lies in $\mathscr A(\widetilde X, S\otimes E)^\Gamma$ (see Proposition $\ref{prop:schwartz}$ above). Now the same proof as in \cite[Theorem 6.1]{MR3454548} gives the above theorem.

\begin{remark}
The assumption that  the conjugacy class $\langle h \rangle $ has polynomial growth is only used to guarantee that  the trace map $\tr_h \colon \mathbb C\Gamma \to \mathbb C$ induces a linear map $\tr_h\colon K_0(C_r^\ast(\Gamma)) \to \mathbb C.$ On the other hand, the $L^2$-Lefschetz fixed point theorem continues to hold in complete generality, without any growth conditions on $\langle h \rangle$. Indeed, observe that $\ind_\Gamma(D_E)$ can be represented by elements with finite propagation. Now the local index formula can be calculated by using finite propagation speed methods such as those employed in  \cite[Theorem 3.4]{MR1254310}.  
\end{remark}

As an immediate consequence of the theorem above, we have the following corollary. 

\begin{definition}
	Let $\mathbb Q_\Gamma$ be the field extension of $\mathbb Q$ by the following set of roots of unity:
	\[  \{e^{\pi i/n} \mid \textup{ there exists $\alpha \in \Gamma$ such that the order of $\alpha$ is $n$}\}. \]
\end{definition}

\begin{corollary}\label{cor:alg}
	With the same notation as above, $\tr_h(\ind_\Gamma(D_E))$ is an algebraic number in $\mathbb Q_\Gamma$. If in addition $h$ has infinite order, then $\tr_h(\ind_\Gamma(D_E)) =0$.  
\end{corollary} 
\begin{proof}
	Suppose $h$ has finite order $n$, then the possible values for $\theta$ and $\beta$ that appear in the formula  $\eqref{eq:loc}$ are  $(2k\pi/n)$ for $0\leq k < [n/2]$. It follows that the possible values for $\theta/2$ and $\beta/2$  are $(k\pi/n)$ for $0\leq k < [n/2]$.  This proves the first part.	Now  if $h$ has infinite order, then the fixed point set of $h$ is empty, since the action of $\Gamma$ on $X$ is proper. It follows that $\tr_h(\ind_\Gamma(D_E)) =0$. 
\end{proof}

 \nocite{MR1707352, MR3590536}

\end{document}